\documentclass[11pt]{article}
\newcommand{\Cay}{\mathsf{Cay}}
\usepackage{todonotes}
\usepackage{enumitem}

\newenvironment{poc}{\begin{proof}[Proof of claim]}{\end{proof}}

\usepackage{hyperref}
\usepackage{amsmath, amssymb, amsthm, bm, mathtools, enumitem, caption}

\usepackage{adjustbox}

\usepackage{graphicx, tikz, geometry}
\usepackage{microtype}

\usepackage{tikz}
\usetikzlibrary{positioning,calc,arrows.meta}
\usetikzlibrary{shapes.geometric}

\newtheorem{theorem}{Theorem}[section]
\newtheorem{lemma}[theorem]{Lemma}
\newtheorem{fact}[theorem]{Fact}
\newtheorem{corollary}[theorem]{Corollary}
\newtheorem{proposition}[theorem]{Proposition}

\newtheorem*{conjecture*}{Conjecture}
\newtheorem{problem}[theorem]{Problem}
\newtheorem{claim}[theorem]{Claim}

\theoremstyle{definition}
\newtheorem{definition}[theorem]{Definition}

\theoremstyle{remark}

\DeclarePairedDelimiter\abs{\lvert}{\rvert}

\DeclarePairedDelimiter\ceil{\lceil}{\rceil}

\newcommand{\R}{\mathbb R}

\newcommand{\eps}{\varepsilon}

\newcommand{\EE}{\mathbb E}

\newcommand{\PP}{\mathbb P}
\newcommand{\TT}{\mathbb T}

\newcommand{\cI}{\mathcal I}

\newcommand{\cL}{\mathcal L}

\newcommand{\ba}{{\bm a}}
\newcommand{\bv}{{\bm v}}

\newcommand{\bx}{{\bm x}}
\newcommand{\by}{{\bm y}}
\newcommand{\bz}{{\bm z}}

\newcommand{\CC}{\mathbb C}
\newcommand{\FF}{\mathbb F}

\newcommand{\RR}{\mathbb R}
\newcommand{\SSS}{\mathbb S}
\newcommand{\ZZ}{\mathbb Z}
\renewcommand{\abs}[1]{\left\vert #1 \right \vert}

\newcommand{\norm}[1]{\lVert #1 \rVert}

\usepackage[noabbrev,capitalize]{cleveref}

\crefname{theorem}{Theorem}{Theorems}
\crefname{lemma}{Lemma}{Lemmas}
\crefname{corollary}{Corollary}{Corollaries}
\crefname{equation}{}{}

 \usepackage{xcolor}

\DeclareMathOperator{\dist}{dist}

\title{Chromatic thresholds for linear equations and recurrence}
\author{Hong Liu\thanks{Extremal Combinatorics and Probability Group (ECOPRO), Institute for Basic Science (IBS),  Daejeon, South Korea. Supported by the Institute for Basic Science (IBS-R029-C4). Email: \texttt{hongliu@ibs.re.kr}.}
\and Zhuo Wu\thanks{Departament de Matemàtiques, Universitat Politècnica de Catalunya (UPC),
Carrer de Pau Gargallo 14, 08028 Barcelona, Spain. Z. Wu acknowledges the bilateral AEI+DFG research project PCI2024-155080-2: SRC-ExCo – Structure, Randomness and Computational Methods in Extremal Combinatorics, and the PID2023-147202NB-I00 (COCOA: COntemporary COmbinatorics and its Applications), all funded by MICIU/AEI/10.13039/501100011033. Email: \texttt{zhuo.wu@upc.edu}.}
\and Ningyuan Yang\thanks{School of Mathematical Sciences, Fudan University, Shanghai, China, and Extremal Combinatorics and Probability Group (ECOPRO), Institute for Basic Science (IBS), Daejeon, South Korea. Email: \texttt{nyyang23@m.fudan.edu.cn}.}
\and Shengtong Zhang\thanks{Department of Mathematics, Stanford University, CA, USA. Email: \texttt{stzh1555@stanford.edu}.}}

\date{}

\begin{document}
\maketitle
\begin{abstract}
Let $\cL:\sum_{i=1}^k c_i x_i=0$ be a homogeneous linear equation with $k\ge3$. For an $\cL$-solution-free set $A\subseteq\FF_p$, we study how arithmetic avoidance constrains the global colorability of the Cayley graph $\Cay(\FF_p,A)$. We define the \emph{chromatic threshold} $\delta_\chi(\cL)$ as the infimum of the densities that force $\chi(\Cay(\FF_p,A))$ to be bounded uniformly over the prime $p$. We prove that
\[
\delta_\chi(\cL)=0
\quad\Longleftrightarrow\quad
\text{some subcollection of at least three coefficients of $\cL$ sums to zero.}
\]
This criterion is strictly intermediate between Roth's density criterion, which requires all coefficients to sum to zero, and the Ramsey--Tur\'an criterion, which requires merely a nonempty zero-sum subcollection. In particular, a canceling pair alone is insufficient.

The positive direction combines supersaturation for the balanced subequation with Fourier analysis and a Bohr-set coloring. For the converse, we construct dense solution-free generators with unbounded chromatic number. The main new ingredient is a quantitative lower bound for Cayley graphs on $\ZZ_p^n$ generated by a Hamming ball around the all-ones vector, valid for every prime $p$. We obtain it by introducing a generalized Kneser graph that embeds into the Cayley graph and applying an equivariant Borsuk--Ulam type obstruction. This answers a question of Griesmer and places the classification in the broader hierarchy of measurable, topological, and Bohr recurrence.
\end{abstract}

\section{Introduction}

Roth-type theorems ask when positive density forces a prescribed linear pattern. In this paper we study a finer question: when a dense set avoids the pattern, how complicated can its difference geometry still be? We measure this complexity by using the set as the generator of a Cayley graph. A proper coloring of $\Cay(\FF_p,A)$ is precisely a partition of $\FF_p$ into classes $V$ satisfying $((V-V)\setminus\{0\})\cap A=\varnothing$; thus bounded chromatic number is a global structural conclusion, substantially stronger than the existence of one large difference-avoiding set.

Fix a homogeneous linear equation
$\cL:\sum_{i=1}^k c_i x_i=0,$ where $k\ge3$ and $c_1,\dots,c_k\in\ZZ\setminus\{0\}.$
A set $A\subseteq\FF_p$ is \emph{$\cL$-solution-free} if it contains no solution with pairwise distinct coordinates. The basic density problem was settled by Roth: positive density forces such a solution exactly for translation-invariant equations.

\begin{theorem}[\cite{roth1954certain}]\label{thm:roth}
Let $\cL:\sum_{i\in[k]}c_i x_i=0$ be a homogeneous linear equation with $k\ge 3$ and $c_1,\dots,c_k\in\ZZ\setminus\{0\}$. Let further $p$ be a prime. Then the following are equivalent:
\begin{itemize}
\item[(a)] Every $\cL$-solution-free set $A\subseteq \FF_p$ has size $o(p)$.
\item[(b)] $\sum_{i\in[k]}c_i=0$.
\end{itemize}
\end{theorem}

Roth's theorem is a sharp zero-density classification, but it does not distinguish among the many possible structures of positive-density solution-free sets. Our aim is to insert a chromatic layer between density forcing and partition regularity. See \cite{croot2017progression,green2005szemeredi,meshulam1995subsets} for representative results on linear equations in finite abelian groups.

\subsection{Roth meets Erd\H{o}s and Simonovits}

Studying $\cL$-solution-free sets can be viewed as an additive analogue of Tur\'an-type problems in extremal graph theory, which ask for the largest size of an $F$-free graph. It is thus natural to import graph-theoretic perspectives. The first such refinement of Roth's problem is Ramsey--Tur\'an in nature. For an abelian group $\Gamma$ and $A\subseteq\Gamma$, let $\Cay(\Gamma,A)$ be the Cayley digraph with vertex set $\Gamma$ and an arc $u\to v$ when $u\ne v$ and $v-u\in A$; all graph parameters below refer to its underlying simple graph. Buci\'c, Christoph, Kim, Lee, and Sivashankar \cite{bucić2025ramseyturanvariantrothstheorem} classified the equations for which a solution-free set is either sparse or leaves a linear-size independent set in its Cayley graph.

\begin{theorem}[\cite{bucić2025ramseyturanvariantrothstheorem}]\label{thm:bckls}
Let $\cL:\sum_{i\in[k]}c_i x_i=0$ be a homogeneous linear equation with $k\ge 3$ and $c_1,\dots,c_k\in\ZZ\setminus\{0\}$. Let further $p$ be a prime. Then the following are equivalent:
\begin{itemize}
\item[(a)] Every $\cL$-solution-free set $A\subseteq \FF_p$ with $\alpha(\Cay(\FF_p,A))=o(p)$ has size $o(p)$.
\item[(b)] There exists a non-empty subset of coefficients of $\cL$ whose sum is zero.
\end{itemize}
\end{theorem}

A linear-size independent set is only a one-cell certificate; bounded chromatic number asks for a bounded partition of the entire ambient group into independent sets. This is the additive analogue of the chromatic-threshold problem of Erd\H{o}s and Simonovits \cite{erdos1973valence}. For a graph $H$, its chromatic threshold is
\begin{align*}
\delta_\chi(H)\coloneqq
 \inf \{d>0:~ &\exists~C=C(H,d)~\text{ such that for any $n$-vertex $H$-free graph $G$},\\
& \text{if $\delta(G)\ge dn$, then $\chi(G)\le C$}\}.
\end{align*}
This notion measures how much density is required, under an $H$-free constraint, to force bounded chromatic number. In a remarkable work, Allen, B{\"o}ttcher, Griffiths, Kohayakawa, and Morris \cite{allen2013chromatic} determined this parameter for every graph $H$; when $\chi(H)=r\ge3$, it is one of
$\frac{r-3}{r-2}$,  $\frac{2r-5}{2r-3}$, $\frac{r-2}{r-1}$.
See \cite{balogh2016hypergraphs,bourneuf2025denseneighborhoodlemmaapplications,kim2026stabilityminusculestructurechromatic,liu2024beyond,wu2025edgedensityminimumdegree} for related developments.

We transfer this question to linear equations by replacing minimum degree with the density of the Cayley generator.

\begin{definition}
Let $\cL:\sum_{i\in[k]}c_i x_i=0$ be a homogeneous linear equation with $k\ge 3$ and $c_1,\dots,c_k\in\ZZ\setminus\{0\}$. The \emph{chromatic threshold of $\cL$} is
\begin{align*}
\delta_\chi(\cL)\coloneqq
\inf \Bigl\{d >0:~ &\exists~C=C(\cL,d)~\text{such that for every prime $p$ and every}\\
&\text{$\cL$-solution-free $A\subseteq\FF_p$ with $|A|\ge dp$, one has }
\chi\bigl(\Cay(\FF_p,A)\bigr)\le C \Bigr\}.
\end{align*}
\end{definition}

Our main result gives a complete classification of the equations for which this threshold vanishes.

\begin{theorem}\label{thm:main0}
Let $\cL:\sum_{i\in[k]}c_i x_i=0$ be a homogeneous linear equation with $k\ge 3$ and $c_1,\dots,c_k\in\ZZ\setminus\{0\}$. Then the following are equivalent:
\begin{itemize}
\item[(a)] $\delta_\chi(\cL)=0$; equivalently, every $\cL$-solution-free set $A\subseteq \FF_p$ with $\chi(\Cay(\FF_p,A))=\omega(1)$ has size $o(p)$. 
\item[(b)] There exists a subset of coefficients of $\cL$ of size at least three whose sum is zero.
\end{itemize}
\end{theorem}

A notable feature of \cref{thm:main0} is that the dichotomy is governed by a zero-sum subcollection of size at least \emph{three}: a canceling \emph{pair} of coefficients alone does not force bounded chromatic number. Thus, \cref{thm:main0} pinpoints an intermediate regime between the Roth and Ramsey--Tur\'an settings: the class of equations with vanishing chromatic threshold lies strictly between the Roth-degenerate and Ramsey--Tur\'an-degenerate classes, as illustrated in Figure~\ref{figure 1}.

The conclusion is genuinely global. For every fixed density $\varepsilon>0$, a dense $\cL$-solution-free set in the vanishing-threshold regime yields a partition
\[
\FF_p=V_1\sqcup\cdots\sqcup V_C,
\qquad C=C(\varepsilon,\cL),
\qquad ((V_i-V_i)\setminus\{0\})\cap A=\varnothing
\]
for every $i$. 

\begin{figure}[ht]
\centering
{
\begin{tikzpicture}

\node[draw=black, fill=blue!15, rectangle,
      rounded corners=1.5cm,
      minimum width=12cm, minimum height=4.6cm,
      align=center, inner sep=0pt] (outer) at (0,0) {};

\node[draw=black, fill=green!15, rectangle,
      rounded corners=1.0cm,
      minimum width=9cm, minimum height=3.2cm,
      align=center, inner sep=0pt] (middle) at (-0.6,-0.4) {};

\node[draw=black, fill=yellow!15, rectangle,
      rounded corners=0.5cm,
      minimum width=6cm, minimum height=1.8cm,
      align=center, inner sep=0pt] (inner) at (-1.2,-0.8) {};

\node[align=center] at ([shift={(0.6,1.8)}]outer.center) {
    Ramsey--Tur\'an criterion\\
    e.g.\ $x_1-x_2+3x_3=0$
};

\node[align=center] at ([shift={(0.3,1.0)}]middle.center) {
    $\delta_{\chi}(\cL)=0$\\
    e.g.\ $x_1-2x_2+3x_3-4x_4=0$
};

\node[align=center] at ([shift={(0,0)}]inner.center) {
    Roth criterion\\
    e.g.\ $x_1-3x_2+2x_3=0$
};

\end{tikzpicture}
}
\caption{The strict hierarchy between Roth's density criterion, vanishing chromatic threshold, and the Ramsey--Tur\'an criterion.}
\label{figure 1}
\end{figure}

\subsection{Applications in topological dynamics}

The high-chromatic Cayley graphs required for the converse direction arise naturally in recurrence theory. This connection both motivates the construction and leads to a second main contribution of the paper.
Let $\Gamma$ be a discrete abelian group. Following Katznelson \cite{katznelson}, we say that a set $S\subset\Gamma$ is
\begin{enumerate}[label=\textup{$(\mathrm{R}\arabic*)$}, ref={$(\mathrm{R}\arabic*)$}]
\item\label{it:density_recurrent} \emph{measurably recurrent} if for every $A\subseteq\Gamma$ of positive upper Banach density, $(A-A)\cap S\ne\varnothing$;
\item\label{it:topological_recurrent} \emph{topologically recurrent} if $\Cay(\Gamma,S)$ has infinite chromatic number;
\item\label{it:bohr_recurrent} \emph{Bohr recurrent} if the complement of $S$ contains no Bohr set.
\end{enumerate}
One has \ref{it:density_recurrent} $\Rightarrow$ \ref{it:topological_recurrent} $\Rightarrow$ \ref{it:bohr_recurrent}. The independence-number and chromatic-number conditions in \cref{thm:bckls,thm:main0} are finite-group analogues of the first two recurrence notions.

The reverse implication \ref{it:bohr_recurrent}$\Rightarrow$\ref{it:topological_recurrent} is Katznelson's question, originating in work of Veech \cite{veech1968equicontinuous} and F{\o}lner \cite{folner1954note}. It remains open in general; see \cite{katznelson,griesmer2023special,alweiss2025new} for background and recent developments.

In the other direction, Bergelson \cite{bergelson1987ergodic}, Furstenberg \cite{furstenberg1981recurrence}, and Ruzsa \cite{ruzsa1982uniform} asked whether \ref{it:topological_recurrent} implies \ref{it:density_recurrent} in $\ZZ$. K\v{r}\'i\v{z} \cite{kvrivz1987large} and Ruzsa \cite{ruzsa1985difference} answered negatively using Cayley graphs on $\ZZ_2^n$ generated by Hamming balls around the all-ones vector. Their construction simultaneously exhibits large chromatic number and a positive-density independent set.

Griesmer \cite{mathoverflow_hamming_chromatic} asked whether the high-chromatic part of this construction persists in odd characteristic. We answer this affirmatively, with a quantitative estimate.\footnote{A similar statement was claimed by Forrest in his thesis \cite{forrest1990recurrence}, but the argument on pp.~144--146 appears to be incorrect.} Write $\mathrm d(\cdot,\cdot)$ for Hamming distance.

\begin{theorem}\label{thm:cayley}
Let $p$ be a prime, $n$ a positive integer, and
$S\coloneqq\{\bx\in\ZZ_p^n:\mathrm d(\bx,\bm1)\le p\sqrt n\}.$
Then
\[
\chi\bigl(\Cay(\ZZ_p^n,S)\bigr)\ge\frac{\sqrt n}{p^3}.
\]
Moreover, as $n\to\infty$,
$\alpha\bigl(\Cay(\ZZ_p^n,S)\bigr)=\Omega_p(|\ZZ_p^n|).$
\end{theorem}

The first conclusion resolves Griesmer's question, while the second retains the positive-density independent-set feature needed in recurrence applications. The proof introduces a generalized Kneser graph adapted to $\ZZ_p$, embeds it into the Cayley graph, and derives its chromatic lower bound from a $\ZZ_p$-equivariant Borsuk--Ulam type obstruction.

For $p=2$, the corresponding phenomenon underlies the constructions of K\v{r}\'i\v{z} and Ruzsa. In an upcoming work~\cite{separatingnote}, we combine \cref{thm:cayley} with additional dynamical machinery to separate~\ref{it:topological_recurrent} from~\ref{it:density_recurrent} in every countably infinite abelian group.\footnote{Since that argument is not part of the present paper, we use the recurrence discussion here only to explain the origin and significance of the finite-dimensional Cayley-graph theorem.}

\begin{theorem}[\cite{separatingnote}]\label{cor:separation}
Let $\Gamma$ be a countably infinite abelian group. Then there exists a subset of $\Gamma$ that is topologically recurrent but not measurably recurrent.
\end{theorem}

From this perspective, the contrapositive of $(a)\Rightarrow(b)$ in~\cref{thm:main0} gives a finite-group analogue of the same phenomenon. Even under the additional constraint of avoiding solutions to a fixed linear equation $\mathcal L$ (for example $x+ry=z$ with $r\neq0$), there exist dense sets that are topologically recurrent but not measurably recurrent. \cref{thm:cayley} supplies the obstruction needed for this nonvanishing direction of \cref{thm:main0}.

Conversely, the implication $(b)\Rightarrow(a)$ in~\cref{thm:main0} gives a stronger structural conclusion: 
dense $\cL$-solution-free sets are not Bohr recurrent: the proof produces a bounded-rank Bohr neighborhood
$B$ for which $|A\cap B|=O_{\cL}(1)$; discretizing the defining characters of $B$ then yields the bounded coloring. Thus, $(b)\Rightarrow(a)$ can be viewed as a finite-group analogue of the classical implication that topological recurrence~\ref{it:topological_recurrent} implies Bohr recurrence~\ref{it:bohr_recurrent}. 

\subsection{Related results and our approach}

\paragraph{Position among classical criteria.} For a single homogeneous equation, Rado's theorem \cite{rado1933studien} says that partition regularity is equivalent to the existence of a nonempty zero-sum subcollection of coefficients; injective partition regularity is equivalent to partition regularity \cite{graham1990ramsey}. Roth's density criterion is the stronger requirement that all coefficients sum to zero. Theorem~\ref{thm:main0} identifies the exact intermediate requirement: a zero-sum subcollection must exist, and it must contain at least three coefficients. Thus the chromatic threshold is not merely an interpolation by analogy; it selects a new algebraic class strictly between density regularity and partition regularity.

This work also belongs to a broader program of importing graph-theoretic extremal notions into arithmetic, including analogues of Sidorenko's property and commonness; see \cite{dongli2024uncommon,foxphamzhao2019common,kml2023sidorenko,kml2024uncommon,saadwolf2017ramsey}.

\paragraph{Our approach.}
The two directions of \cref{thm:main0} require different structural mechanisms. We describe them in the order in which they appear in the paper.

\noindent\textit{The Fourier--Bohr direction: $(b)\Rightarrow(a)$.} Assume that $A\subseteq \FF_p$ is $\cL$-solution-free and has positive density. Condition $(b)$ provides a translation-invariant sub-equation in at least three variables. By a supersaturation theorem, a dense set $A$ contains many solutions to this sub-equation. This leads to a structural restriction: $A$ cannot place many points inside the Bohr set $B$ associated with its large spectrum. Intuitively, points in this Bohr set behave like approximate periods for the structured part of $A$; if $A$ clustered there, then the shift-invariance forced by the zero-sum sub-equation would produce too many completions to a forbidden solution of $\cL$. Once we know that $A\cap B$ is small, a bounded coloring follows: discretizing the phases of the characters in the large spectrum partitions $\FF_p$ into a bounded number of cells in which all differences lie in $B$, so each cell induces a sparse subgraph of $\Cay(\FF_p,A)$ and hence has bounded chromatic number.  The point where the assumption that the zero-sum subcollection has at least three coefficients becomes essential
is in the Fourier control of the contribution from frequencies outside the large spectrum: having at least three terms allows us to dominate higher moments by the $L^2$-norm via Parseval identity, which is precisely what fails when the only zero-sum subcollection has size two.

\noindent\textit{The construction direction: the contrapositive of $(a)\Rightarrow(b)$.}
If there is no zero-sum subcollection at all, the Ramsey--Tur\'an result already yields the counterexamples needed to show $\delta_\chi(\cL)>0$. The genuinely new case is when zero-sum pairs exist but no zero-sum subcollection has size at least three. After relabeling, one may assume $c_1+c_2=0$. We then construct positive-density $\cL$-solution-free generators whose chromatic numbers are arbitrarily large. The argument has two stages.

\noindent\textit{Stage 1: a topological obstruction inside a Cayley graph (Section~\ref{sec:gen-kneser}).}
To build intuition and highlight the main difficulties, we begin with a guiding example: the Schur equation $x-y+z=0$. In this case there is a classical high-chromatic obstruction inside a Cayley graph on $\ZZ_2^n$, namely the Kneser graph $\mathrm{KN}(n,k)$, whose vertices are the $k$-subsets of $[n]$ with edges between disjoint pairs. Lov\'asz~\cite{lovasz1978kneser} famously proved that $\chi(\mathrm{KN}(n,k)) = n-2k+2$. The graph embeds naturally into $\ZZ_2^n$ by encoding a $k$-set $A\subseteq[n]$ as its indicator vector $\bm{1}_A$, so disjointness corresponds to a difference vector concentrated near the all-ones vector. Moreover, a small Hamming neighborhood of $\bm 1$ in $\ZZ_2^n$ is itself $\{x-y+z=0\}$-solution-free, so the same generator both supports the embedding (hence large chromatic number) and satisfies the required solution-freeness. 

The binary construction, however, is not robust enough for the general
equations arising here. For example, the equation
$x_1-x_2+x_3+x_4=0,$
which has a canceling pair but no zero-sum subcollection of size at least
three, reduces over $\ZZ_2$ to
$x_1+x_2+x_3+x_4=0$. A Hamming ball of radius two around $\bm1$ already
contains four distinct solutions: take
$\bm1$, $\bm1+\bm e_1$, $\bm1+\bm e_2$, and
$\bm1+\bm e_1+\bm e_2$.
Thus the same Hamming neighborhood cannot simultaneously support the desired high-chromatic obstruction and remain $\cL$-solution-free. More fundamentally, the norm-separation argument in Stage~2 requires an auxiliary prime larger than the coefficient scale, forcing us to work over $\ZZ_p^n$ for arbitrary primes $p$. In odd characteristic, however, the classical indicator-vector embedding of the Kneser graph no longer interacts correctly with differences near $\bm1$. The main task of Stage~1 is therefore to construct a genuinely $p$-ary Kneser-type graph, together with a compatible embedding into a Hamming-ball Cayley graph on $\ZZ_p^n$, that can subsequently be combined with the arithmetic construction of $\cL$-solution-free generators.

To this end, we introduce $\mathrm{KN}(n,k,p-1)$, whose vertices are ordered $(p-1)$-tuples of disjoint $k$-sets and whose prefix--suffix adjacency is designed so that
\[
(A_1,\dots,A_{p-1})
\longmapsto
\bm1_{A_1}+2\bm1_{A_2}+\cdots+(p-1)\bm1_{A_{p-1}}
\]
embeds the graph into a Hamming-ball Cayley graph on $\ZZ_p^n$.

The chromatic lower bound uses a free $\ZZ_p$-action on an odd-dimensional sphere. A generic family of labeled $\ZZ_p$-orbits yields $p$ disjoint cyclic sectors, each containing at least $k$ labels. Given a proper coloring, for each color $j$ we let $U_j$ consist of the sphere points whose sectors realize a vertex of color $j$. The sets $U_j$ form an open cover. Dold's equivariant obstruction forces a complete $\ZZ_p$-orbit into one $U_j$; witnesses at two consecutive orbit points satisfy the prefix--suffix adjacency condition and therefore form a monochromatic edge. This proves the generalized Kneser bound and hence the chromatic estimate in \cref{thm:cayley}. The linear-size independent set in that theorem is obtained separately by a two-dimensional Berry--Esseen estimate.

\noindent\textit{Stage 2: adding arithmetic separation and density (Section~\ref{sec:a-implies-b}).}
The Hamming-ball generator alone need not be $\cL$-solution-free. We first work in a product group
$\ZZ_m\cong\prod_{i=1}^n\ZZ_{p_i}$ and define an asymmetric coordinate norm centered near $p_i/q$. Its near-maximal level set $E_0$ contains a discretized copy of the generator from Stage~1, so its Cayley graph has large chromatic number; at the same time, a norm-separation inequality prevents the linear form $\sum_i c_i x_i$ from vanishing on $E_0^k$.

The set $E_0$ is not dense enough. We therefore construct a positive-density extension set $F_0$ whose two relevant norms are simultaneously small. A two-dimensional Berry--Esseen theorem gives $|F_0|=\Theta(m)$, while the norm gap ensures that the $F_0$-contribution cannot cancel the $E_0$-contribution. Finally, we lift both sets to a sufficiently large prime cyclic group. The absence of zero-sum coefficient subcollections of size at least three reduces every possible mixed solution to the canceling pair $c_1+c_2=0$, which is excluded by the extension property.

\noindent\textbf{Organization.}
Section~\ref{sec:preliminaries} collects the analytic, probabilistic, and topological tools. Section~\ref{sec:b-implies-a} proves the Fourier--Bohr direction. Section~\ref{sec:gen-kneser} develops the generalized Kneser graph and proves \cref{thm:cayley}. Section~\ref{sec:a-implies-b} constructs the dense high-chromatic solution-free sets, and Section~\ref{sec:conclude} records further questions and related thresholds.

\section{Preliminaries}\label{sec:preliminaries}

We first collect notations that are used in this paper. 

For an integer $m$, denote $[m]$ to be the set $\{1,\dots,m\}$, and $\ZZ_m$ to be the abelian group $\ZZ/m\ZZ$. When there is no ambiguity, sometimes we identify $\{0,1,\dots,m-1\}$ with $\ZZ_m$ and write the indices additively $\bmod m$. For $a\in\ZZ_m$, we denote by $\overline{a}$ its representative in $\{0,1,\dots,m-1\}$.  For the sake of presentation, we will omit floors and ceilings whenever they are not important.

Throughout the paper, $p,q$ denote prime numbers, and $\FF_p$ denotes the finite field of order $p$ (it may be identified with the cyclic group $\ZZ/p\ZZ$ under addition when convenient). We reserve $\ZZ_p$ for the cyclic group of order $p$ when discussing group actions/topology. In vector spaces, boldface symbols denote vectors; in particular, for a set $A$, $\bm{1}_A$ denotes its characteristic vector that takes value $1$ on $A$ and $0$ elsewhere. For two sets $A,B$ in a field, denote $A+B\coloneqq\{a+b:a\in A,b\in B\}$. For a set $A$ and a scalar $k\in\ZZ$, denote $kA\coloneqq\{ka:a\in A\}$.

Expectation and probability are denoted by $\EE$ and $\PP$, respectively. For random variables $\xi_1,\xi_2$, we write
$\mathrm{Cov}(\xi_1,\xi_2)\coloneqq\EE[(\xi_1-\EE \xi_1)(\xi_2-\EE \xi_2)]$ and
$\mathrm{Var}(\xi_1)\coloneqq\mathrm{Cov}(\xi_1,\xi_1)$.
For a bivariate random variable $X=(\xi_1,\xi_2)\in\RR^2$, let the covariance matrix be
\[
\mathrm{Cov}(X)\coloneqq
\begin{bmatrix}
    \mathrm{Var}(\xi_1) & \mathrm{Cov}(\xi_1,\xi_2)\\
    \mathrm{Cov}(\xi_2,\xi_1) & \mathrm{Var}(\xi_2)
\end{bmatrix}.
\]
For a symmetric matrix $M$, $\lambda_{\min}(M)$ and $\lambda_{\max}(M)$ denote its
smallest and largest eigenvalues, respectively.

We write $N(\mu,\Sigma)$ for the Gaussian distribution with mean $\mu$ and covariance $\Sigma$.
The standard normal distribution function and density function are denoted respectively by
\[
\Phi(t)\coloneqq\frac{1}{\sqrt{2\pi}}\int_{-\infty}^t e^{-x^2/2}\,dx,
\qquad
\varphi(t)\coloneqq\frac{1}{\sqrt{2\pi}}e^{-t^2/2}.
\]

For $d\ge 0$, $\SSS^d$ denotes the unit sphere in $\RR^{d+1}$.
For a nonempty set $A\subseteq \SSS^d$, $\dist(x,A)$ denotes the geodesic
distance from $x$ to $A$.

For prime $p$ and a topological space $X$, a continuous map $g\colon X\to X$ is called an \emph{$\ZZ_p$-action} if $g^p=\mathrm{id}$. It is called \emph{free} if there is no $x$ with $gx=x$. Given $\ZZ_p$-actions $g_X$ on $X$ and $g_Y$ on $Y$, a map $f\colon X\to Y$ is called \emph{$\ZZ_p$-equivariant} (or simply equivariant if $\ZZ_p$ is fixed) if $f(g_Xx)=g_Yf(x)$ for all $x\in X$.

\subsection{Topology}

In this subsection, we collect several basic topological notions that will be used later in the paper. Our purpose is not to provide a comprehensive introduction to topology, but rather
to present a small collection of basic topological facts that will be used in our arguments.

Fix an odd prime $p$, and write $\zeta \coloneqq e^{2\pi i/p}$ for a primitive $p$-th root
of unity.
We identify the real vector space $\mathbb R^{2m}$ with $\mathbb C^m$ in the
standard way, by grouping the $2m$ real coordinates into $m$ complex pairs.

Under this identification, the unit sphere
\[
\SSS^{2m-1} = \{ \bx \in \mathbb C^m : \|\bx\| = 1 \}
\]
is preserved by complex scalar multiplication of modulus one. Here we write $\norm{\cdot}$ for the Euclidean norm.
In particular, the cyclic group $\mathbb Z_p$ acts freely and continuously on $\SSS^{2m-1}$ by rotating every complex coordinate by angle $2\pi/p$:
\[
g_X \bx \coloneqq \zeta \bx,
\]
where $g$ denotes a fixed generator of $\mathbb Z_p$.

The following result is the main topological theorem that we will use in this paper. It is a corollary of Dold's Theorem \cite{1983Simple}, whose idea has been used for Kneser-type results in \cite{3e87df0e-4a01-3fa8-9594-dbda78a6d23c,ziegler2002generalized}. For a comprehensive treatment on this topic, see \cite{matouvsek2003using}.

\begin{theorem}\label{thm:Zp-LS}
Let $U_1,\dots,U_{t+1}$ be closed sets that cover $\SSS^{2m-1}$. If $t(p-1) < 2m$, then at least one $U_i$ contains a full $\mathbb Z_p$--orbit $\{\bx,\zeta \bx,\dots,\zeta^{p-1}\bx\}$.
\end{theorem}

\begin{proof}

Let
\[
H \coloneqq \left\{(y_0,\dots,y_{p-1})\in\R^p : \sum_{a=0}^{p-1} y_a = 0\right\}
\]
be a subspace of $\R^p$, then $\dim H = p-1$.
Let $\mathbb Z_p$ act on $\mathbb R^p$ by cyclic permutation of coordinates:
\[
g_Y(y_0,y_1,\dots,y_{p-1}) = (y_1,y_2,\dots,y_{p-1},y_0).
\]
Obviously this action preserves $H$. Moreover, it has no nonzero fixed vector in $H$:
indeed, $g_Y\cdot \by=\by$ forces all coordinates equal, hence $\sum y_a=0$ implies $\by=0$.

For $t \ge 1$, define
\[
V \coloneqq H^t \cong \mathbb R^{t(p-1)}, \qquad
Y \coloneqq \SSS^{t(p-1)-1}\subseteq V.
\]
We equip $V$ with the \emph{diagonal} $\mathbb Z_p$--action. For
$\bv = (\bv_1,\dots,\bv_t) \in H^t$, define
\[
g_Y \bv \coloneqq (g_Y \bv_1,\dots,g_Y \bv_t).
\]

This action is free on $Y$.
Indeed, suppose that $g_Y \bv = \bv$ for some
$\bv \in Y$.
Then, $g_Y \bv_i = \bv_i$ for each $i$, which forces
$\bv_i = 0$ for all $i$, a contradiction.

For each $j=1,\dots,t$ define the map $\Phi_j:\SSS^{2m-1}\to \mathbb R^p$ by
\[
\Phi_j(\bx)_a \coloneqq \dist(\zeta^a \bx, U_j) - \dist(\zeta^{a+1}\bx, U_j),
\qquad a\in\{0,1,\dots,p-1\},
\]
where indices are taken mod $p$.
Continuity of $\Phi_j$ follows from continuity of $\dist(\cdot,U_j)$.

Summing over $a$ gives a telescoping sum:
\[
\sum_{a=0}^{p-1} \Phi_j(\bx)_a
=
\sum_{a=0}^{p-1}\bigl(\dist(\zeta^a \bx, U_j) - \dist(\zeta^{a+1}\bx, U_j)\bigr)=0,
\]
so $\Phi_j(\bx)\in H$ for every $\bx\in \SSS^{2m-1}$ and $j\le t$.

Now define
\[
f:\SSS^{2m-1}\to V=H^t,\qquad f(\bx)\coloneqq(\Phi_1(\bx),\dots,\Phi_t(\bx)).
\]

Then for each $j$ and each $a$,
\[
\Phi_j(\zeta \bx)_a
=
\dist(\zeta^a(\zeta \bx),U_j)-\dist(\zeta^{a+1}(\zeta \bx),U_j)
=
\dist(\zeta^{a+1}\bx,U_j)-\dist(\zeta^{a+2}\bx,U_j)
=\Phi_j(\bx)_{a+1},
\]
which is exactly the cyclic permutation action of $g_Y$ on coordinates. Thus each block $\Phi_j$ is equivariant, hence the map $f$ is also equivariant.

Assume for contradiction that none of $U_1,\dots,U_{t+1}$ contains a full
$\mathbb Z_p$--orbit.
We first show that $f$ has no zeros.
Otherwise, let $\bx\in \SSS^{2m-1}$ satisfy $f(\bx)=0$.
Then for each $j\le t$ and each $a\in\{0,1,\dots,p-1\}$ we have
\[
0=\Phi_j(\bx)_a=\dist(\zeta^a \bx,U_j)-\dist(\zeta^{a+1}\bx,U_j),
\]
so the $p$ numbers
\[
\dist(\bx,U_j),~\dist(\zeta \bx,U_j),~\dots,~\dist(\zeta^{p-1}\bx,U_j)
\]
are all equal.

Since the sets $U_1,\dots,U_{t+1}$ cover the sphere, there exists an index
$i$ with $\bx\in U_i$.
If $i\le t$, then $\dist(\bx,U_i)=0$, and by the above equality of distances we
obtain $\dist(\zeta^a\bx,U_i)=0$ for all $a$.
Because $U_i$ is closed, $\dist(y,U_i)=0$ implies $y\in U_i$, and hence
$\zeta^a\bx\in U_i$ for every $a$.
Thus $U_i$ contains the full orbit of $\bx$, a contradiction.
Therefore $\bx\in U_{t+1}$.

Since $f$ is equivariant and $f(\bx)=0$, we have $f(\zeta^a\bx)=f(g_X^{(a)}\bx)=g_Y^{(a)}f(\bx)=0$ for each $a\in\{0,1,\dots,p-1\}$. Now, repeat the same reasoning for each point $\zeta^a\bx$ in the orbit, we know that $\zeta^a\bx\in U_{t+1}$ for every $a$.
This shows that $U_{t+1}$ contains the full orbit of $\bx$, again a contradiction. Therefore $f$ has no zeros.

To complete the argument, we invoke the following standard consequence of Dold's theorem \cite{1983Simple}.

\begin{theorem}[Dold's Theorem for $\ZZ_p$]\label{thm:dold}
Let $X$ be a $d$-connected space with a $\ZZ_p$-action, and let $Y$ be a space with dimension at most $d$ and a free $\ZZ_p$-action.
Then there is no $\ZZ_p$-equivariant map $X\to Y$.
\end{theorem}

In our application, $X$ will be an odd-dimensional sphere. Recall that $\SSS^{2m-1}$ is $(2m-2)$-connected, so we may take $d=2m-2$.

As $f$ has no zeros, it takes values in $V\setminus\{0\}$. Composing $f$ with the radial retraction yields a $\ZZ_p$--equivariant map
\[
\SSS^{2m-1}\xrightarrow{\,f\,} V\setminus\{0\} \xrightarrow{\,r\,} \SSS^{t(p-1)-1},
\]
where $r$ is the radial projection onto the unit sphere in $V$,
\[
r(\bm{v}) \coloneqq \frac{\bm{v}}{\|\bm{v}\|}\qquad (\bm{v}\in V\setminus\{0\}).
\]
The map $r$ is $\ZZ_p$--equivariant because the $\ZZ_p$--action on $V$ is
orthogonal and hence preserves the Euclidean norm. Moreover, the induced action on $\SSS^{t(p-1)-1}$ is free, and by the inequality
$t(p-1)<2m$ we have
\[
\dim\bigl(\SSS^{t(p-1)-1}\bigr)=t(p-1)-1 \le 2m-2 = d.
\]
Therefore Theorem~\ref{thm:dold} applies and forbids the existence of such an
equivariant map, a contradiction.
\end{proof}

According to the Shrinking Lemma (see, e.g., \cite{MunkresTopology}), any finite open cover $\{U_1, \dots, U_t\}$ of a metric space admits a closed cover $\{C_1, \dots, C_t\}$ with $C_i \subseteq U_i$ for all $i$. The following corollary is immediate.

\begin{corollary}\label{coro:Zp-LS}
Let $U_1,\dots,U_{t+1}$ be open sets that cover $\SSS^{2m-1}$. If $t(p-1) < 2m,$ then one of the sets
$U_1,\dots,U_t,U_{t+1}$ contains a full $\mathbb Z_p$--orbit
$\{\bx,\zeta \bx,\dots,\zeta^{p-1}\bx\}$.
\end{corollary}

\subsection{Additive combinatorics}

Our proof also uses Fourier analysis over the finite field $\FF_p$.
Fix a prime $p$. For a function $f:\FF_p\to\CC$, define its Fourier transform by
\[
\widehat f(\xi) \coloneqq \frac{1}{p}\sum_{x\in\FF_p} f(x)\,e_p(-x\xi),
\qquad \xi\in\FF_p,
\]
where $e_p(t) = e^{2\pi i t/p}$.
We will repeatedly use the following basic facts.

\begin{fact}\label{fact:fourier}
The following standard facts hold.
\begin{enumerate}
    \item \textrm{Fourier inversion.}
    For every $x\in\FF_p$, $f(x)  = \sum_{\xi\in\FF_p} \widehat f(\xi)\, e_p(\xi x)$.

    \item \textrm{Parseval identity.}
    For every $f:\FF_p\to\CC$, $\sum_{x\in\FF_p} |f(x)|^2
     = 
    p\sum_{\xi\in\FF_p} |\widehat f(\xi)|^2$.
\end{enumerate}
\end{fact}

As a corollary, we have the following solution--counting formula.
\begin{lemma}
    \label{lem:solution--counting}
    Let $A\subseteq \FF_p$, and let $c_1,\dots,c_k\in\FF_p$ be coefficients. For each fixed $y\in\FF_p$,
    the number of solutions to the linear equation
    $$\sum_{i = 1}^k c_i x_i = y, \quad\text{where }~(x_1, \cdots, x_k) \in A^k$$
    is given by
    $$N(y) = p^{k - 1} \sum_{\xi \in \FF_p} \prod_{i = 1}^k \widehat{1_A}(c_i \xi) e_p(y \xi).$$
\end{lemma}

\begin{proof}
By Fourier inversion formula,
$1_{\{y\}}(x)=\frac1p\sum_{\xi\in\FF_p} e_p((x-y)\xi)$.
Therefore,
\[
N(y)= \sum_{(x_1,\dots,x_k)\in A^k}
1_{\{y\}}\!\left(\sum_{i=1}^k c_i x_i\right)
=\frac1p\sum_{\xi\in\FF_p} e_p(-y\xi)
\prod_{i=1}^k \sum_{x_i\in A} e_p(c_i x_i\xi).
\]
By the definition of the Fourier transform, this equals
\[
p^{k-1}\sum_{\xi\in\FF_p}
\Bigl(\prod_{i=1}^k \widehat{1_A}(-c_i\xi)\Bigr)e_p(-y\xi),
\]
and replacing $\xi$ by $-\xi$ yields the desired formula.
\end{proof}

We also need a supersaturated version of Roth's theorem. This follows from the arithmetic removal lemma of Green~\cite{green2005szemeredi}, which reads as follows.
\begin{theorem}
    \label{thm:ARL}
    Let $k \geq 3$ be a fixed integer. For any $\eps > 0$, there exists $\delta > 0$ such that the following holds.
    Let $G$ be an abelian group with size $N$, and suppose that $A_1, \cdots, A_k$ are subsets of $G$
    such that there are at most $\delta N^{k - 1}$ solutions to the equation $a_1 + \cdots + a_k = 0$ with $a_i \in A_i$ for all $i$. Then one may remove at most $\eps N$ elements from each $A_i$ to obtain subsets $A_i'$ with no solutions
    to $a_1 + \cdots + a_k = 0$ with $a_i \in A_i'$ for all $i$.
\end{theorem}

\begin{corollary}
    \label{cor:Roth--supersaturation}
    Let $\cL:\sum_{i=1}^k c_i x_i=0$ be a homogeneous linear equation with $k\ge 3$, and $\sum_{i = 1}^k c_i = 0$.
    For any $\eps > 0$, there exists $\delta = \delta(\eps, \cL) > 0$ such that the following holds.
    For any prime $p$ and any $A \subseteq \FF_p$ with size at least $\eps p$,  there exist at least $\delta p^{k - 1}$ solutions to $\cL$ in $A^k$.
\end{corollary}

\begin{proof}
    By setting $\delta$ sufficiently small, the result trivially holds when $p \le \max_i |c_i|$. Thus we may assume that $p > \max_i |c_i|$.
    Apply \cref{thm:ARL} with $\eps' = \frac{\eps}{2k}$ to the sets $A_i = c_i A$. If the equation $\cL$ has at most $\delta p^{k - 1}$ solutions in $A^k$, then the equation
    $$x_1 + \cdots + x_k = 0$$
    has at most $\delta p^{k - 1}$ solutions in $A_1 \times \cdots \times A_k$. The arithmetic removal lemma shows that we can remove $\eps' p$ elements from each $A_i$ to obtain sets $A_i'$, such that  $A_1' \times \cdots \times A_k'$ has no solution to $x_1 + \cdots + x_k = 0$. This is impossible, since for each $a \in A$ we must remove at least one element in $(c_1 a, \cdots, c_k a)$, which means that we must remove at least $\frac{\eps}{k} p$ elements from one $A_i$, a contradiction.
\end{proof}

Finally, we state a simple lemma that upper bounds the number of solutions of a linear equation. 

\begin{lemma}
    \label{thm:upper--bound}
    Let $\cL:\sum_{i=1}^k c_i x_i= y$ be a linear equation with nonzero coefficients.
    Then the number of solutions to $\cL$ in $\FF_p^k$ is at most $p^{k - 1}$.
    Moreover, the number of solutions with at least two equal coordinates is at most $k^2 p^{k - 2}$, with the only exception equivalent to $x_1-x_2=0$.
\end{lemma}

\subsection{Probabilistic methods}

This subsection collects the probabilistic ingredients used later.
We need a two-dimensional Berry--Esseen bound over convex sets, which is a special case of \cite[Theorem 1.1]{bentkus2005lyapunov}.

\begin{lemma}\label{lem:2D-BE}
Let $X_1,\dots,X_n$ be independent random vectors in $\mathbb R^2$ and define
\[
X\coloneqq\sum_{i=1}^n X_i,\qquad \mu\coloneqq\EE X,\qquad \Sigma\coloneqq\mathrm{Cov}(X).
\]
Assume that $\Sigma$ is positive definite, let $M = \Sigma^{1/2}$ denote the unique symmetric positive definite square root of $\Sigma$,
and let $Z\sim N(\mu,\Sigma)$.

Then there exists an absolute constant $\beta>0$ such that for every convex set
$A\subseteq\mathbb R^2$,
\[
\big|\mathbb P(X\in A)-\mathbb P(Z\in A)\big|
 \le
\beta\sum_{i=1}^n 
\mathbb E\Big\|M^{-1}\big(X_i-\mathbb EX_i\big)\Big\|^3.
\]
\end{lemma}

We will use it through a convenient corollary tailored to bounded summands
and a well-conditioned covariance matrix.

\begin{corollary}\label{cor:2D-BE-convenient}
In the setting of Lemma~\ref{lem:2D-BE}, assume moreover that
$\|X_i-\mathbb EX_i\|\le 1$ almost surely for all $i$, and that
\[
\lambda_{\min}(\Sigma) \ge \sigma^2 n
\]
for some constant $\sigma>0$ independent of $n$.
Then for every convex set $A\subseteq\mathbb R^2$,
\[
\big|\mathbb P(X\in A)-\mathbb P(Z\in A)\big|
\le
\frac{\beta}{\sigma^3}\cdot \frac1{\sqrt n}.
\]
\end{corollary}

\begin{proof}
Since $\lambda_{\min}(\Sigma)\ge \sigma^2 n$, we have $\lambda_{\max}\left(M^{-1}\right)\le (\sigma\sqrt n)^{-1}$. Hence, using $\|X_i-\mathbb EX_i\|\le 1$ a.s., we have
\[
\Big\|M^{-1}\big(X_i-\mathbb EX_i\big)\Big\|
\le\frac1{\sigma\sqrt n}\quad\text{a.s.},
\]
and therefore
\[
\sum_{i=1}^n \mathbb E\Big\|M^{-1}\big(X_i-\mathbb EX_i\big)\Big\|^3
\le n\cdot \Big(\frac1{\sigma\sqrt n}\Big)^3
=\frac1{\sigma^3\sqrt n}.
\]
The claim follows from Lemma~\ref{lem:2D-BE}.
\end{proof}

\medskip

We will also need a uniform, dimension-free lower bound for certain Gaussian rectangle
probabilities when the covariance is well-conditioned at scale $n$. Recall that $\Phi$ and $\varphi$ denote the standard normal distribution function and density function respectively.

\begin{lemma}\label{lem:gauss-rect-const}
Let $n\ge 1$ and let $\Sigma$ be a $2\times 2$ positive definite matrix such that
\[
\Sigma=
\begin{bmatrix}
    \Sigma_{11} & \Sigma_{12}\\
    \Sigma_{12} & \Sigma_{22}
\end{bmatrix},
\quad cn\le \lambda_{\min}(\Sigma)\le \lambda_{\max}(\Sigma)\le Cn,
\]
for some constants $0<c\le C$ independent of $n$.
Let $Z\sim N(\mu,\Sigma)$ where $\mu=(\mu_1,\mu_2)$ and set $t\coloneqq r\sqrt n$ with $r\ge 1$. Then
\[
\mathbb P\big(Z\in(-\infty,\mu_1-t]\times(-\infty,\mu_2-t]\big) \ge \alpha_r(c,C),
\]
where one may take the explicit constant
\[
\alpha_r(c,C)
\coloneqq
\big(\Phi(a)-\Phi(2a)\big)\cdot 
\Phi\!\Big(\frac{a(1+2\rho_0)}{\sqrt{1-\rho_0^2}}\Big),
\qquad
a\coloneqq-\frac{r}{\sqrt c},
\qquad
\rho_0\coloneqq\sqrt{1-\Big(\frac{c}{C}\Big)^2}.
\]
\end{lemma}

\begin{proof}
Let $Z=(Z_1,Z_2)\sim N(\mu,\Sigma)$ where $\mu=(\mu_1,\mu_2)$. Write $\xi_1\coloneqq(Z_1-\mu_1)/\sqrt{\Sigma_{11}}$ and $\xi_2\coloneqq(Z_2-\mu_2)/\sqrt{\Sigma_{22}}$.
Then $(\xi_1,\xi_2)$ is centered Gaussian with unit variances and correlation
$\rho\coloneqq\mathrm{Cov}(\xi_1,\xi_2)=\Sigma_{12}/\sqrt{\Sigma_{11}\Sigma_{22}}$.

First, from $cn\le \Sigma_{jj}\le Cn$ and $t=r\sqrt n$, we get
\[
-\frac{t}{\sqrt{\Sigma_{jj}}} \ge-\frac{r\sqrt n}{\sqrt{cn}} = a,
\qquad j=1,2.
\]
Hence
\[
\mathbb P\big(Z_1\le \mu_1-t,~Z_2\le \mu_2-t\big)
=\mathbb P\Big(\xi_1\le -\frac{t}{\sqrt{\Sigma_{11}}},~\xi_2\le -\frac{t}{\sqrt{\Sigma_{22}}}\Big)
\ge \mathbb P(\xi_1\le a,~\xi_2\le a).
\]

Second, we bound $|\rho|$ using the eigenvalue constraints. Since
$\det(\Sigma)=\lambda_{\min}(\Sigma)\lambda_{\max}(\Sigma)\ge c^2n^2$
and $\Sigma_{11}\Sigma_{22}\le (\lambda_{\max}(\Sigma))^2\le C^2n^2$, we have
\[
\rho^2 \;=\; 1-\frac{\det(\Sigma)}{\Sigma_{11}\Sigma_{22}}
\;\le\; 1-\Big(\frac{c}{C}\Big)^2,
\quad\text{so}\quad
|\rho|\le \rho_0.
\]

Finally, using the standard conditional representation for a bivariate Gaussian,
for any $\rho\in[-\rho_0,\rho_0]$,
\[
\mathbb P(\xi_1\le a,~\xi_2\le a)
=\int_{-\infty}^a \Phi\!\Big(\frac{a-\rho x}{\sqrt{1-\rho^2}}\Big)\varphi(x)\,dx
 \ge \int_{2a}^a \Phi\!\Big(\frac{a-\rho x}{\sqrt{1-\rho^2}}\Big)\varphi(x)\,dx.
\]
For $x\in[2a,a]$ and $|\rho|\le\rho_0$, one checks
\[
\frac{a-\rho x}{\sqrt{1-\rho^2}}
 \ge \frac{a(1+2\rho_0)}{\sqrt{1-\rho_0^2}}.
\]
Therefore,
\[
\mathbb P(\xi_1\le a,~\xi_2\le a)
\ge
\big(\Phi(a)-\Phi(2a)\big)\cdot 
\Phi\!\Big(\frac{a(1+2\rho_0)}{\sqrt{1-\rho_0^2}}\Big)
=\alpha_r(c,C),
\]
as claimed.
\end{proof}

We also need the following fact for eigenvalue bounds under summation.

\begin{fact}\label{fact:eig-sum}
Let $A_1,\dots,A_n$ be $2\times 2$ symmetric positive semidefinite matrices. Then
\[
\lambda_{\min}\Big(\sum_{i=1}^n A_i\Big) \ge \sum_{i=1}^n \lambda_{\min}(A_i),
\qquad
\lambda_{\max}\Big(\sum_{i=1}^n A_i\Big) \le  \sum_{i=1}^n \lambda_{\max}(A_i).
\]
Consequently, if $X_1,\dots,X_n$ are independent random vectors in $\mathbb R^2$ and
$X=\sum_{i=1}^n X_i$, then
\[
\mathrm{Cov}(X)=\sum_{i=1}^n \mathrm{Cov}(X_i),
\]
and the same eigenvalue inequalities hold with $A_i=\mathrm{Cov}(X_i)$.
\end{fact}

\begin{proof}
For any unit vector $u\in\mathbb R^2$, by the Rayleigh quotient,
\[
u^\top\Big(\sum_{i=1}^n A_i\Big)u=\sum_{i=1}^n u^\top A_i u
\ge \sum_{i=1}^n \lambda_{\min}(A_i),
\]
and taking the infimum over $\|u\|=1$ gives the first inequality.
Similarly,
\[
u^\top\Big(\sum_{i=1}^n A_i\Big)u=\sum_{i=1}^n u^\top A_i u
\le \sum_{i=1}^n \lambda_{\max}(A_i),
\]
and taking the supremum over $\|u\|=1$ yields the second inequality.
\end{proof}



\section{Proof of~\cref{thm:main0} $(b) \Rightarrow (a)$}\label{sec:b-implies-a}

In this section we prove~\cref{thm:main0} in the direction $(b) \Rightarrow (a)$. Unpacking the definitions, it suffices to prove the following statement.

\begin{theorem}\label{thm:b-implies-a-reduction}
Let $\cL:\sum_{i=1}^k c_i x_i=0$ be a homogeneous linear equation with $k\ge 3$ and
$c_1,\dots,c_k\in\ZZ\setminus\{0\}$. Assume that there exists $I\subseteq [k]$ with $|I|\ge 3$
and $\sum_{i\in I} c_i=0$.
Then for every $\varepsilon>0$, there exists $C=C(\varepsilon,\cL)$ such that for any $\cL$-solution-free set $A\subseteq \FF_p$, if $|A|\ge \varepsilon p$, then 
$$\chi\!\left(\Cay(\FF_p,A)\right)\le C.$$
\end{theorem}

\begin{proof} 
By taking $C$ sufficiently large, we may assume that $p > \max_i |c_i|$, so all the coefficients of $\cL$ are nonzero in $\FF_p$.

If $I = [k]$, then by Corollary \ref{cor:Roth--supersaturation} and Lemma \ref{thm:upper--bound}, there are at least $\delta p^{k - 1}$ solutions to $\cL$ in $A^k$, and at most $k^2 p^{k - 2}$ of those can have duplicate coordinates. Therefore, we have $p < k^2 \delta^{-1}$, so taking $C = k^2 \delta^{-1}$ gives the desired result. 

Now suppose $I \neq [k]$, and let $J = [k] \backslash I$. Consider the homogeneous linear equation
$$\cL_1: \sum_{i \in I} c_i x_i = 0.$$
By Corollary \ref{cor:Roth--supersaturation}, there exists $\delta > 0$ depending only on $\eps$ and $\cL_1$ such that $A^I$ contains at least $\delta p^{|I| - 1}$ solutions to $\cL_1$. Take $\nu = \frac{\delta}{6}$, and define the large spectrum
\[
L \coloneqq \Bigl\{\xi\in \FF_p:~\bigl|\widehat{1_A}(\xi)\bigr| \ge \nu \Bigr\}.
\]
By Parseval's identity (Fact \ref{fact:fourier}(2)), we have
\[
\sum_{\xi\in \FF_p} \bigl|\widehat{1_A}(\xi)\bigr|^2  = \frac{|A|}{p}  \le  1.
\]
Hence we have $\bigl|\widehat{1_A}(\xi)\bigr|\le 1$ for all $\xi\in\FF_p$, and $|L|\le \nu^{-2}$. We take an element $s \in I$, and set
\[
\Gamma\coloneqq c_s^{-1} L=\{\eta\in \FF_p:~c_s\eta\in L\}.
\]
Then $|\Gamma| = |L|\le \nu^{-2}$.

For $x \in \FF_p$, define
\[
\norm{x}_{\tau}\coloneqq\|x/p\|_{\TT}= \min_{n\in\mathbb{Z}} |x/p-n|
\]
as the distance from $x / p$ to its closest integer. We now set
$$D \coloneqq \sum_{i \in J} |c_i|, \quad \quad\rho \coloneqq \frac{1}{6\pi} D^{-1} \nu^2 \delta = \frac{1}{216\pi} D^{-1} \delta^3,$$
and consider the Bohr set
$$B = B(\Gamma, \rho) = \{x \in \FF_p: \norm{\xi x}_{\tau} \leq \rho, \forall \xi \in \Gamma\}.$$
Our key claim is the following.

\begin{claim}\label{cl:AcapB}
If $p > 6k^2 \delta^{-1}$, we have $\abs{A \cap B} < k$.
\end{claim}

\begin{poc}
Suppose for the sake of contradiction that $\abs{A \cap B} \geq k$. Then we can find $\{x_i: i\in J\}$ such that the $x_i$'s are pairwise distinct, and $x_i \in A \cap B$ for each $i \in J$. Set
$$y = -\sum_{i \in J}c_i x_i.$$
Let $N_1, N_2$ denote the number of solutions in $A^I$ to the equations
$$\sum_{i \in I} c_i x_i = 0\quad \text{and} \quad \sum_{i \in I} c_i x_i = y$$
respectively. By assumption, we have $N_1 \geq \delta p^{|I| - 1}$. Furthermore, by Lemma \ref{lem:solution--counting}, we have
$$N_1 = p^{|I| - 1} \sum_{\xi \in \FF_p} \prod_{i \in I}\widehat{1_A}(c_i \xi)\quad \text{and}\quad N_2 = p^{|I| - 1} \sum_{\xi \in \FF_p} \prod_{i \in I} \widehat{1_A}(c_i \xi) e_p(y \xi).$$
Thus, we have
$$\abs{N_1 - N_2} \leq p^{|I| - 1}  \sum_{\xi \in \FF_p} \prod_{i \in I} \abs{\widehat{1_A}(c_i \xi)} \abs{1 - e_p(y \xi)}.$$
We split the sum based on whether $\xi$ lies in $\Gamma$. If $\xi \in \Gamma$, we use the trivial bound $ \abs{\widehat{1_A}(c_i \xi)} \leq 1$ for each $i$, and use the triangle inequality to bound
$$\abs{1 - e_p(y \xi)}=2|\sin(\pi y \xi)| \leq 2\pi \norm{y \xi}_{\tau} \leq 2\pi\sum_{i \in J} \norm{c_i x_i \xi}_{\tau} \leq 2\pi\sum_{i \in J} |c_i| \norm{x_i \xi}_{\tau}.$$
As $\norm{x_i \xi}_{\tau} \leq \rho$ by assumption $x_i\in B$, we obtain
$$\abs{1 - e_p(y \xi)} \leq 2\pi\sum_{i \in J} \norm{c_i x_i \xi}_{\tau}\le2\pi \sum_{i \in J} |c_i|  \rho \leq 2\pi D \rho \leq \frac{1}{3}\nu^2 \delta.$$
Hence, we obtain
$$\sum_{\xi \in \Gamma} \prod_{i\in I} \abs{\widehat{1_A}(c_i \xi)} \abs{1 - e_p(y \xi)} \leq \abs{\Gamma} \cdot \frac{1}{3}\nu^2 \delta \leq \frac{1}{3} \delta.$$

If $\xi \notin \Gamma$, then we trivially bound $\abs{1 - e_p(y \xi)} \leq 2$. By the definition of $\Gamma$, we have $\abs{\widehat{1_A}(c_s \xi)} \leq \nu$ in this case, so we obtain
$$\prod_{i \in I} \abs{\widehat{1_A}(c_i \xi)} \leq \nu \prod_{i \in I \backslash \{s\}} \abs{\widehat{1_A}(c_i \xi)}.$$
Thus by H\"{o}lder's inequality, we have
$$\sum_{\xi \notin \Gamma} \prod_{i \in I} \abs{\widehat{1_A}(c_i \xi)} \abs{1 - e_p(y \xi)} \leq 2 \nu \sum_{\xi \in \FF_p} \prod_{i \in I \backslash \{s\}} \abs{\widehat{1_A}(c_i \xi)} \leq 2 \nu\prod_{i \in I \backslash \{s\}} \left(\sum_{\xi \in \FF_p} \abs{\widehat{1_A}(c_i \xi)}^{|I|- 1}\right)^{1 / (|I| - 1)}.$$
For each $i \in I \backslash \{s\}$, by Parseval's identity and the assumption $|I| \geq 3$, we have
$$\sum_{\xi \in \FF_p} \abs{\widehat{1_A}(c_i \xi)}^{|I|- 1} \leq \sum_{\xi \in \FF_p} \abs{\widehat{1_A}(c_i \xi)}^{2} \leq 1.$$
So we conclude that
$$\sum_{\xi \notin \Gamma} \prod_{i \in I} \abs{\widehat{1_A}(c_i \xi)} \abs{1 - e_p(y \xi)} \leq 2 \nu \leq \frac{1}{3} \delta.$$
Summing both components together, we obtain
$$\abs{N_1 - N_2} \leq \frac{2}{3} \delta p^{|I| - 1},$$
so we have
$$N_2 \geq \frac{1}{3} \delta p^{|I| - 1}.$$

Therefore, the number of solutions to
$$\sum_{i \in I} c_i x_i = y$$
in $A^I$ is at least $\frac{1}{3} \delta p^{|I| - 1}$. By Lemma \ref{thm:upper--bound}, the number of solutions in $A^I$ with two equal coordinates is at most $k^2 p^{|I| - 2}$, and the number of solutions in $A^I$ with at least one coordinate in $\{x_j: j \in J\}$ is also at most $k^2 p^{|I| - 2}$. We conclude that, when $p > 6k^2 \delta^{-1}$, there is at least one solution to $\sum_{i \in I} c_i x_i = y$ in $A^I$ with pairwise distinct coordinates, and all coordinates distinct from  $\{x_j: j \in J\}$. Thus we have
$$\sum_{i \in I} c_i x_i = y = - \sum_{i \in J} c_i x_i$$
so we have
$$\sum_{i = 1}^k c_i x_i = 0$$
and the $x_i$'s are pairwise distinct elements of $A$. This contradicts the assumption that $A$ is $\cL$--solution--free. 
\end{poc}

We now set
$$C = C(\eps, \cL) \coloneqq \max\left( 6k^2 \delta^{-1}, 2k \ceil{2\rho^{-1}}^{\nu^{-2}}\right).$$
It is straightforward to check that $C$ depends only on $\eps$ and $\cL$. We now check that the chromatic number of $\Cay(\FF_p,A)$ is bounded by $C$.

If $p \leq 6k^2 \delta^{-1}$, then we have
$$\chi\!\left(\Cay(\FF_p,A)\right)\le p \le 6k^2 \delta^{-1} \leq C.$$
Now assume that $p > 6k^2 \delta^{-1}$. Let $M\coloneqq\lceil 2\rho^{-1}\rceil$ and partition the unit circle $\SSS^1$ into arcs
$\mathcal I_0,\dots,\mathcal I_{M-1}$ of equal length $2\pi/M$.
Define a partition $\kappa:\FF_p\to [M]^{\Gamma}$ as follows.
For each $u\in\FF_p$ and $\xi\in\Gamma$, let $\kappa_\xi(u)\in\{0,\dots,M-1\}$ be the unique index such that
$e_p(\xi u)\in \mathcal I_{\kappa_\xi(u)}$, and set
\[
\kappa(u)\coloneqq\bigl(\kappa_\xi(u)\bigr)_{\xi\in\Gamma}.
\]
Therefore the number of parts satisfies
\[
|\mathrm{im}(\kappa)| \le  M^{|\Gamma|}
 \le  \bigl\lceil 2\rho^{-1}\bigr\rceil^{\nu^{-2}}.
\]
For each $\ba=(a_\xi)_{\xi\in\Gamma} \in \mathrm{im}(\kappa)$, let $V_\ba = \kappa^{-1}(\ba)$. Then $V_\ba$, $\ba\in \mathrm{im}(\kappa)$, partition $\FF_p$. 
Let $\Cay(\FF_p,A)[V_\ba]$ be the subgraph of $\Cay(\FF_p,A)$ induced by $V_\ba$. If $v_2$ is an in-neighbor of $v_1$ in $\Cay(\FF_p,A)[V_\ba]$, then $v_1-v_2\in A$; and also
$$e_p(\xi v_1),e_p(\xi v_2)\in \cI_{a_{\xi}}\implies \norm{\xi (v_1 - v_2)}_{\tau} \leq \frac{2}{M}\le \rho,\quad\forall \xi\in\Gamma.$$
This entails $v_1 - v_2 \in A \cap B$. By Claim \ref{cl:AcapB}, we have $\abs{A \cap B} < k$, so each $v_1 \in V_{\ba}$ has at most $(k-1)$ in-neighbors. Symmetrically, each $v_1 \in V_{\ba}$ has at most $(k-1)$ out-neighbors. We conclude that $\Cay(\FF_p,A)[V_\ba]$ has maximum degree at most $2(k - 1)$, hence by Brooks' theorem (see \cite{brooks1941colouring}) we have
$$\chi\left(\Cay(\FF_p,A)[V_\ba]\right) \leq 2k - 1.$$
Summing over $\ba$, we conclude that
$$\chi\left(\Cay(\FF_p,A)\right) \leq \sum_{\ba \in \mathrm{im}(\kappa)}\chi\left(\Cay(\FF_p,A)[V_\ba]\right) \leq (2k - 1)\bigl\lceil 2\rho^{-1}\bigr\rceil^{\nu^{-2}} \leq C$$
as desired.
\end{proof}

\section{Generalized Kneser graphs}\label{sec:gen-kneser}

The proof of the implication $(a) \Rightarrow (b)$ in
Theorem~\ref{thm:main0} requires explicit constructions of dense sets whose
associated Cayley graphs have large chromatic number. In this section, we introduce a variant of the Kneser graph (Definition~\ref{def:gen-kneser}) tailored to our
additive setting. Our main result in this section is a chromatic lower bound on such Kneser graphs obtained by topological methods, see \cref{thm:chi-lb}. The key feature of this construction is that it admits a natural embedding into the underlying undirected graph of a Cayley graph on $\ZZ_{m+1}^n$ generated by a Hamming ball around the all-ones vector, allowing us to prove \cref{thm:cayley}.

\subsection{A generalized Kneser graph and its chromatic number}\label{sec:gen-Kneser-defn}

We begin by defining the generalized Kneser graph that will serve as our main
combinatorial object. Roughly speaking, vertices encode \emph{ordered}
collections of pairwise disjoint $k$-subsets of $[n]$. The ordering is essential:
it allows us to impose a cyclic structure that will later correspond to addition
in $\ZZ_{m+1}$ and to the geometry of Hamming balls.

\begin{definition}\label{def:gen-kneser}
Fix integers $n,m,k\ge 1$ such that $n\ge(m+1)k$. We define the
\emph{generalized Kneser graph} $\mathrm{KN}(n,k,m)$ as follows. Let
\[
V(\mathrm{KN}(n,k,m))\coloneqq
\left\{(A_1,\dots,A_m)\in\binom{[n]}{k}^{m}:\;
A_i\cap A_j=\varnothing~\text{for all }i\neq j\right\}.
\]
For vertices $\mathcal A=(A_1,\dots,A_m)$ and
$\mathcal B=(B_1,\dots,B_m)$ we declare $\mathcal A\sim\mathcal B$ if either of
the following holds:
\begin{enumerate}[label=\textup{$(\mathrm{A}\arabic*)$.}, ref={$(\mathrm{A}\arabic*)$}]
\item\label{it:adj1}
$\left(\bigcup_{\ell=1}^i A_\ell\right)\cap
\left(\bigcup_{\ell=i}^{m} B_\ell\right)=\varnothing$
for all $i\in[m]$;
\item\label{it:adj2}
$\left(\bigcup_{\ell=1}^i B_\ell\right)\cap
\left(\bigcup_{\ell=i}^{m} A_\ell\right)=\varnothing$
for all $i\in[m]$.
\end{enumerate}
\end{definition}

Note that in the special case $m=1$, this definition recovers the classical Kneser graph $\mathrm{KN}(n,k)$.

The adjacency conditions may be interpreted as prefix--suffix disjointness
requirements: condition~\ref{it:adj1} asserts that every cumulative prefix of
$\mathcal A$ avoids the complementary suffix of $\mathcal B$, while
condition~\ref{it:adj2} imposes the same requirement with the roles of
$\mathcal A$ and $\mathcal B$ reversed. Together, they ensure that adjacency is
symmetric, while encoding a directional cyclic structure that will later
correspond to addition in $\ZZ_{m+1}$.

The main result of this section is the following lower bound on the
chromatic number of $\mathrm{KN}(n,k,m)$.

\begin{theorem}\label{thm:chi-lb}
Let $n,m,k\ge 2$ such that $n\ge(m+1)k$ and $m+1=p$ is a prime. Then
\[
\chi(\mathrm{KN}(n,k,m)) > \frac{\frac{n}{p}-k}{p(p-1)}.
\]
\end{theorem}

\subsection{Embedding into Cayley graphs: Proof of~\cref{thm:cayley}}

Recall that the $p=2$ case was observed by Kříž~\cite{kvrivz1987large} and Ruzsa~\cite{ruzsa1985difference}. From now on we suppose $p\geq3$.

We first prove the lower bound on the chromatic number. Let $k= np^{-1}-\sqrt{n}p^{-1}$. We will embed $G=\mathrm{KN}(n,k,p-1)$ into the underlying undirected graph of $\Cay(\ZZ_p^n,S)$, and hence the lower bound \cref{thm:chi-lb} on $\chi(G)$ transfers to the Cayley graph. Therefore, one has
    \[
    \chi\!\left(\Cay(\ZZ_p^n,S)\right) \ge \chi(G) > \frac{\sqrt{n}}{p^3}.
    \]
    
    To this end, we associate each vertex $\mathcal{A}$ of $G$ to a vector in $\ZZ_p^n$ as
    \[
    \bx_\mathcal{A}=1\cdot\bm{1}_{A_1}+\dots+(p-1)\cdot\bm{1}_{A_{p-1}}.
    \]
    Fix an edge $\mathcal A\mathcal B\in E(G)$, where
    $\mathcal A=(A_1,\dots,A_{p-1})$ and $\mathcal B=(B_1,\dots,B_{p-1})$.
    Write
    \[
    A_0\coloneqq[n]\setminus\bigcup_{i=1}^{p-1} A_i
    \quad\text{and}\quad
    B_0\coloneqq[n]\setminus\bigcup_{i=1}^{p-1} B_i.
    \]
    Assume that condition~\ref{it:adj1} holds.  We first prove the following observation.

    \begin{claim}\label{ob:inter}
        We have $|A_0\cap B_{p-1}|=k$ and $|A_i\cap B_{i-1}|\geq(p+1)k-n$
        \footnote{The argument holds also when $(p+1)k-n$ is negative.}
        for all $i\in[p-1]$. 
    \end{claim}

    \begin{poc}
        We know that $B_{p-1}\subseteq A_0$ since $B_{p-1}\cap\left(\bigcup_{\ell=1}^{p-1}A_\ell\right)=\varnothing$. This implies that $|A_0\cap B_{p-1}|=|B_{p-1}|=k$. Similarly, $A_1\subseteq B_0$ and thus $|A_1\cap B_0|=k\geq(p+1)k-n$.

        For $2\leq i\leq {p-1}$, it suffices to prove that $|A_i\setminus B_{i-1}|\leq n-pk$. By condition~\ref{it:adj1}, we have
        \[
        \bigcup_{\ell=1}^iA_\ell\subseteq\bigcup_{\ell=0}^{i-1}B_\ell,\qquad
        \bigcup_{\ell=1}^{i-1}A_\ell\subseteq\bigcup_{\ell=0}^{i-2}B_\ell.
        \]

        hence
        \[
        \left(\bigcup_{\ell=1}^{i-1}A_\ell\right)\sqcup(A_i\setminus B_{i-1})\subseteq\bigcup_{\ell=0}^{i-2}B_{\ell}.
        \]
        We have
        \[
        |A_i\setminus B_{i-1}|\leq\left|\bigcup_{\ell=0}^{i-2}B_{\ell}\right|-\left|\bigcup_{\ell=1}^{i-1}A_\ell\right|=(n-(p-i+1)k)-(i-1)k=n-pk,
        \]
        as desired.
    \end{poc}

    Writing the difference coordinatewise according to the partitions
    $[n]=A_0\sqcup A_1\sqcup\cdots\sqcup A_{p-1}$ and
    $[n]=B_0\sqcup B_1\sqcup\cdots\sqcup B_{p-1}$, we obtain
    \[
    \bx_\mathcal{A}-\bx_\mathcal{B}
    =(0\cdot\bm{1}_{A_0}-(p-1)\cdot\bm{1}_{B_{p-1}})
    +(1\cdot\bm{1}_{A_1}-0\cdot\bm{1}_{B_0})
    +\cdots
    +((p-1)\cdot\bm{1}_{A_{p-1}}-(p-2)\cdot\bm{1}_{B_{p-2}}).
    \]
    In particular, for every coordinate $j\in A_0\cap B_{p-1}$ we have
    \[
    (\bx_\mathcal{A}-\bx_\mathcal{B})_j  = 0-(p-1)  \equiv 1 \pmod{p},
    \]
    and for every $i\in[p-1]$ and every coordinate $j\in A_i\cap B_{i-1}$ we have
    \[
    (\bx_\mathcal{A}-\bx_\mathcal{B})_j  = i-(i-1)  \equiv 1 \pmod{p}.
    \]
    Therefore, on the coordinates in
    \[
    (A_0\cap B_{p-1})\ \cup\ \bigcup_{i=1}^{p-1} (A_i\cap B_{i-1}),
    \]
    the vector $\bx_\mathcal{A}-\bx_\mathcal{B}$ agrees with the all-ones vector $\bm 1$. It follows from Claim~\ref{ob:inter} that
    \[
    \mathrm{d}(\bx_\mathcal{A}-\bx_\mathcal{B},\bm 1)
     \le 
    n-\Bigl(|A_0\cap B_{p-1}|+\sum_{i=1}^{p-1} |A_i\cap B_{i-1}|\Bigr) \le 
    n-\Bigl(k+(p-1)\bigl((p+1)k-n\bigr)\Bigr)
    =
    pn-p^2k.
    \]
    By the choice of $k$, we have $\bx_\mathcal{A}-\bx_\mathcal{B}\in S$, and hence there is an arc from $\bx_\mathcal{B}$ to $\bx_\mathcal{A}$ in the Cayley digraph
    $\Cay(\ZZ_p^n,S)$. Similarly, if condition~\ref{it:adj2} holds, then the same conclusion holds
    with the roles of $\mathcal A$ and $\mathcal B$ reversed. Therefore, $G$ is an underlying subgraph of $\Cay(\ZZ_p^n,S)$. This gives the lower bound on the chromatic number.

It remains to construct an independent set of size linear in $p^n$. We shall pick all the elements $\bx$ that both $\bx$ 
and $-\bx$ are ``far away'' from $\bm1$, and show that the number of such elements is linear in $p^n$ via a 2D Berry--Esseen argument. For an element $\bx=(x_1,\dots,x_n)\in\ZZ_p^n$, let
\[
f(x_i)\coloneqq\begin{cases}
    \frac{p-x_i}{p-1},&x_i\neq0,\\
    0,&x_i=0,
\end{cases}\qquad \text{ and }\qquad 
f(\bx)\coloneqq \sum_{i=1}^nf(x_i),
\]
and let
\[
I\coloneqq\left\{\bx\in\ZZ_p^n:f(\bx)\leq n/2-p\sqrt n,\quad f(-\bx)\leq n/2-p\sqrt n\right\}.
\]
One checks that $f$ satisfies the triangle inequality. Hence $f(\bx-\by)\le f(\bx)+f(-\by)\le n-2p\sqrt n$ for $\bx,\by\in I$.
On the other hand, for $\bz\in S$, one has $f(\bz)\geq n-p\sqrt n$. Thus $(I-I)\cap S=\varnothing$, i.e. $I$ is an independent set.

Let $\by=(y_1,\dots,y_n)$ be uniformly random in $\ZZ_p^n$, and for each $i\in[n]$ define the two-dimensional random vector
\[
X_i\coloneqq (f(y_i),f(-y_i)).
\]
Then,
\[
X\coloneqq X_1+\dots+X_n=(f(\by),f(-\by)).
\]
To apply Corollary \ref{cor:2D-BE-convenient}, we first verify that $\norm{X_i-\EE X_i}\leq1$. Indeed, $\EE(X_i)=(1/2,1/2)$, and $X_i\in[0,1]^2$, thus it always holds that $\norm{X_i-\EE X_i}\leq\sqrt2/2$.
    
Now it suffices to control the covariance matrix of $X$. Let $\Sigma_i$ be the covariance matrix of $X_i$. A direct computation gives
\[
\Sigma_i=
\begin{bmatrix}
    \frac{p+1}{12(p-1)} &
    \frac{5-p}{12(p-1)}\\
    \frac{5-p}{12(p-1)} &
    \frac{p+1}{12(p-1)}
\end{bmatrix}
\quad\text{with eigenvalues}\quad
\lambda_{i,1}=\frac{1}{2(p-1)},\quad
\lambda_{i,2}=\frac{(p-2)}{6(p-1)},
\]
and the two eigenvalues are both positive. Let $\sigma^2=\min\{\lambda_{i,1},\lambda_{i,2}\}$, $\sigma'^2=\max\{\lambda_{i,1},\lambda_{i,2}\}$.

Since the random vectors $X_1,\dots,X_n$ are mutually independent, one has 
\[
\Sigma\coloneqq \mathrm{Cov}(X)=\sum_{i=1}^n \mathrm{Cov}(X_i)
\quad\text{with eigenvalues}\quad
\lambda_{\min}(\Sigma)= \sigma^2n,\quad
\lambda_{\max}(\Sigma)=\sigma'^2n.
\]

Now let $A=\left[-\infty,n/2-p\sqrt n\right]^2$ and let $Z$ be the Gaussian vector 
$Z\sim N(\mu,\Sigma)$, where $\mu=\EE X=(n/2,n/2)$.
By Lemma \ref{lem:gauss-rect-const}, one has
\[
\PP(Z\in A)\geq\alpha\coloneqq\alpha_p\left(\sigma^2,\sigma'^2\right).
\]
Here $\alpha>0$ is a constant independent of $n$.

By Corollary \ref{cor:2D-BE-convenient}, for sufficiently large $n$, we obtain
\[
\PP(X\in A)\geq \PP(Z\in A)-\frac{\beta}{\sigma^3}\frac{1}{\sqrt n}>\frac{\alpha}{2}.
\]
In particular, for a uniformly random $\by\in\ZZ_p^n$, the probability that
$\max\{f(\by),f(-\by)\}\leq n/2-p\sqrt n$ is at least $\alpha/2$, which implies that $|I|=\Theta\!\left(p^n\right)$.  This completes the proof of~\cref{thm:cayley}.

\subsection{Proof of \texorpdfstring{\cref{thm:chi-lb}}{Theorem chi-lb}}

\begin{proof}
Write $p=m+1$ and $G=\mathrm{KN}(n,k,p-1)$. Suppose for a contradiction that
\[
q\coloneqq\chi(G)\le \frac{\frac np-k}{p(p-1)}.
\]
Fix a proper coloring $c\colon V(G)\to[q]$ and put $t=q-1$. Then
$q\le \left\lfloor\frac{\frac np-k}{p(p-1)}\right\rfloor.$
Set
\[
\gamma\coloneqq\frac{\frac np-k}{p}
\qquad\text{and}\qquad
r\coloneqq\left\lfloor\frac{\gamma}{2}\right\rfloor.
\]
Since $q\ge1$, we have $\gamma\ge p-1\ge2$, and hence $r\ge1$. Moreover,
\[
p(k+2pr)\le n
\quad 
\text{and}\quad
t(p-1)\le \gamma-(p-1)\le \gamma-2<2r.
\]

Identify $\mathbb R^{2r}$ with $\mathbb C^r$, and let
$\zeta=e^{2\pi i/p}$. For $\bx\in\SSS^{2r-1}$ and $i\in\ZZ_p$, define
\[
V_{\bx}^{(i)}\coloneqq
\left\{\by\in\SSS^{2r-1}:
\dist(\zeta^i\bx,\by)<\dist(\zeta^{i+1}\bx,\by)
\text{ and }
\dist(\zeta^i\bx,\by)<\dist(\zeta^{i-1}\bx,\by)
\right\}.
\]
These sets are open and pairwise disjoint. Indeed, if
$\langle\by,\bx\rangle_{\mathbb C}=\rho e^{i\theta}$ with $\rho>0$, then
$\by\in V_{\bx}^{(i)}$ precisely when
\[
\theta\in
\left(\frac{2\pi i}{p}-\frac{\pi}{p},
      \frac{2\pi i}{p}+\frac{\pi}{p}\right)
\pmod{2\pi}.
\]
We shall also use the equivariance relations
\[
V_{\zeta\bx}^{(i)}=V_{\bx}^{(i+1)}
\qquad\text{and}\qquad
\by\in V_{\bx}^{(i)}\Longrightarrow
\zeta^a\by\in V_{\bx}^{(i+a)}.
\]
For $i\in\ZZ_p$, let
\[
H_{\bx}^{(i)}\coloneqq
\left\{\by\in\SSS^{2r-1}:
\dist(\zeta^i\bx,\by)=\dist(\zeta^{i+1}\bx,\by)
\right\}.
\]
Each $H_{\bx}^{(i)}$ is an equatorial sphere, and
\[
\SSS^{2r-1}\setminus\bigcup_{i\in\ZZ_p}V_{\bx}^{(i)}
\subseteq
\bigcup_{i\in\ZZ_p}H_{\bx}^{(i)}.
\]
Indeed, outside the union on the right, a closest point to $\by$ in the
orbit $\{\zeta^i\bx:i\in\ZZ_p\}$ is strictly closer than its two
neighbors, and hence $\by$ lies in the corresponding sector.

Choose a set $X\subseteq\SSS^{2r-1}$ of size $k+2pr$ such that distinct
points of $X$ lie in distinct $\ZZ_p$-orbits and no real hyperplane through
$0$ contains more than $2r$ points of $X$. Such a set exists by a generic
choice. Since $p(k+2pr)\le n$, choose distinct points
$Z=\{\bz_1,\dots,\bz_n\}\subseteq\SSS^{2r-1}$ containing
\[
\ZZ_pX\coloneqq\{\zeta^a\by:\by\in X,\ a\in\ZZ_p\}.
\]

Fix $\bx\in\SSS^{2r-1}$. At most $2pr$ points of $X$ lie in
$\bigcup_iH_{\bx}^{(i)}$, so at least $k$ points of $X$ lie in
$\bigcup_iV_{\bx}^{(i)}$. By the equivariance above, the orbit of each such
point meets every $V_{\bx}^{(i)}$ exactly once. Since these orbits are
pairwise disjoint, every $V_{\bx}^{(i)}$ contains at least $k$ points of
$\ZZ_pX$, and hence at least $k$ points of $Z$.

For $\bx\in\SSS^{2r-1}$ and $i\in\ZZ_p$, set
\[
I_{\bx}^{(i)}\coloneqq\{a\in[n]:\bz_a\in V_{\bx}^{(i)}\}.
\]
Thus the sets $I_{\bx}^{(i)}$ are pairwise disjoint and each has size at
least $k$. For every color $j\in[q]$, define
\[
U_j\coloneqq
\left\{\bx\in\SSS^{2r-1}:
\begin{array}{l}
\text{there exists }\mathcal A=(A_1,\dots,A_{p-1})\in V(G)\text{ with }c(\mathcal A)=j,\\[-1mm]
A_i\subseteq I_{\bx}^{(i)}\text{ for every }i\in[p-1]
\end{array}
\right\}.
\]
Each $U_j$ is open, since for a fixed witness the defining condition is a
finite intersection of strict inequalities depending continuously on $\bx$. Moreover, the sets $U_1,\dots,U_q$ cover the sphere:
for any $\bx$, choose $A_i\in\binom{I_{\bx}^{(i)}}{k}$ for
$i\in[p-1]$. Their pairwise disjointness makes
$(A_1,\dots,A_{p-1})$ a vertex of $G$, and its color witnesses membership
of $\bx$ in one of the $U_j$.

By Corollary~\ref{coro:Zp-LS}, with the case $t=0$ being immediate, there
exist $\bx\in\SSS^{2r-1}$ and $j\in[q]$ such that
\[
\{\bx,\zeta\bx,\dots,\zeta^{p-1}\bx\}\subseteq U_j.
\]
Let $\mathcal A=(A_1,\dots,A_{p-1})$ witness $\bx\in U_j$, and let
$\mathcal B=(B_1,\dots,B_{p-1})$ witness $\zeta\bx\in U_j$. Since
$I_{\zeta\bx}^{(\ell)}=I_{\bx}^{(\ell+1)}$, for every $i\in[p-1]$ we have
\[
\left(\bigcup_{\ell=1}^{i}A_\ell\right)
\cap
\left(\bigcup_{\ell=i}^{p-1}B_\ell\right)
\subseteq
\left(\bigcup_{\ell=1}^{i}I_{\bx}^{(\ell)}\right)
\cap
\left(\bigcup_{\ell=i}^{p-1}I_{\bx}^{(\ell+1)}\right)
=\varnothing,
\]
where the superscripts are read modulo $p$. Thus adjacency
condition~\ref{it:adj1} holds, so $\mathcal A\sim\mathcal B$. But both
vertices have color $j$, contradicting the properness of $c$.
\end{proof}

\section{Proof of~\cref{thm:main0} $(a) \Rightarrow (b)$}\label{sec:a-implies-b}

In this section we prove \cref{thm:main0} in the direction $(a) \Rightarrow (b)$. To this end, we shall construct, for any prescribed chromatic number bound $t$, a dense subset
$A\subseteq \FF_p$ that is $\cL$-solution-free but whose associated Cayley graph has
chromatic number at least $t$.
Formally it suffices to prove the following statement.

\begin{theorem}\label{thm:a-implies-b-reduction}
Let $\cL:\sum_{i=1}^k c_i x_i=0$ be a homogeneous linear equation with $k\ge 3$ and
$c_1,\dots,c_k\in\ZZ\setminus\{0\}$. Assume that there does not exist $I\subseteq [k]$ with $|I|\ge 3$
and $\sum_{i\in I} c_i=0$.
Then there exists $\varepsilon>0$, such that for every integer $t$ and sufficiently large $p$, there is an $\cL$-solution-free set $A\subseteq \FF_p$ with $|A|\ge \varepsilon p$ and 
$$\chi\!\left(\Cay(\FF_p,A)\right)\ge t.$$
\end{theorem}

By \cref{thm:bckls}, if there does not exist $I\subseteq [k]$ with $\sum_{i\in I} c_i=0$, then we can find an $\cL$-solution-free set $A\subseteq \FF_p$ with $|A|\ge  \Omega(p)$ with $\alpha(\Cay(\FF_p,A))=o(p)$ and thus $\chi\!\left(\Cay(\FF_p,A)\right)\geq p/o(p)=\omega(1)$. Hence, from now on we may assume that (after a permutation of variables) 
$$c_1+c_2=0.$$ 
Moreover, if $c_i+c_j=0$ for some $i,j$, then $\{i,j\}\cap\{1,2\}\neq\varnothing$, as otherwise one has $\sum_{i\in I} c_i=0$ for $I=\{1,2,i,j\}$.

The idea is to construct such a set first in a fixed cyclic group $\ZZ_m$, where we can use a product decomposition
$\ZZ_m\cong \prod_i \ZZ_{p_i}$ to define a ``Hamming ball-like'' subset with large chromatic number, and then lift the
construction to $\ZZ_p$ for $p$ sufficiently large compared to $m$.

Throughout this section, for brevity, we let $C=c_1+\dots+c_k$ and $D =|c_1|+\dots+|c_k|$. From the conditions, $C$ is a nonzero integer. If $C\leq-1$, we may replace $c_1,\dots,c_k$ by their negatives. Hence we may assume that $C\geq1$.

\subsection{Construction in product group $\ZZ_m$}\label{subsec:construction_in_product_group}

We first prove the following lemma, which handles groups with a product structure. For two sets $E,F$, we say that $F$ is \emph{an extension of $E$} (in $\ZZ_m$) if
\[
    (-c_1F-c_2F)\cap (c_3E+\dots+c_kE)=\varnothing.
\]

Intuitively, the extension condition says that the contribution from the
$F$-variables can never cancel the contribution from the $E$-variables; in particular, there is
no $\cL$-solution with $x_1,x_2\in F$ and $x_3,\dots,x_k\in E$.

\begin{lemma}
    \label{lem:small-generate-large-cayley}
    Let $\cL$ be an equation from \cref{thm:a-implies-b-reduction}. For any $t > 0$, there exists a large integer $m$ and a subset $E_0 \subseteq \ZZ_m$ with the following properties:
    \begin{enumerate}
        \item $E_0$ is $\cL$-solution-free in $\ZZ_m$.
        \item $\Cay(\ZZ_m,E_0)$ has chromatic number at least $2t$.
        \item There exists an extension $F_0\subseteq\ZZ_m$ of $E_0$ such that $|F_0|=\Theta(m)$.
    \end{enumerate} 
\end{lemma}

\begin{proof}
We shall choose $m$ to be a product of many large primes so that $\ZZ_m$ factors as a product
of cyclic groups. This allows us to introduce a multi-coordinate weight measuring deviation from the
endpoint regions and to construct a high-chromatic generating set $E_0$ defined by a
near-maximal total weight constraint.

    Fix a prime $q\gg D >C$.
    Take $n$ sufficiently large compared to $t,k,q,D $ and take distinct primes $p_1, \cdots, p_n$ sufficiently large compared to $n$.
    Let $m = p_1\cdots p_n$ and set
    \[
    G_0 = \ZZ_m \cong \prod_{i=1}^n \ZZ_{p_i}.
    \]
    For $y\in\ZZ_m$, we identify it with a vector $\left(y^{(1)},\dots,y^{(n)}\right)$, where $y^{(i)}$ is the image of $y$ under the projection $\ZZ_m\to\ZZ_{p_i}$ for each $i\in[n]$.

For each coordinate $i$ and each $j\in[q-1]$, we define a piecewise-linear function of
$\overline{y^{(i)}}\in\{0,\ldots,p_i-1\}$ that measures the distance to the endpoints $0$ and $p_i$, rescaled
asymmetrically with slopes depending on $j$.
More precisely, we define $\norm{\cdot}^{(i)}_j\colon\ZZ_m\to[0,1]$ as
    \[
    \norm{y}^{(i)}_j=\min\left\{\frac{q}{jp_i}\overline{y^{(i)}},
    \frac{q}{(q-j)p_i}\left(p_i-\overline{y^{(i)}}\right)\right\},\quad\forall i\in[n],~j\in[q-1].
    \]
    One readily checks that $\norm{\cdot}^{(i)}_j$ takes value in $[0,1]$ and satisfies the triangle inequality. 
    
    On $\ZZ_m$ we define
    \[
    \norm{y}_j=\sum_{i=1}^n \norm{y}^{(i)}_j.
    \]
    Let
    \[
    E_0 \coloneqq \left\{y \in \ZZ_m: \norm{y}_1\ge n-q\sqrt n-1\right\}.
    \]
Informally, $E_0$ consists of elements for which the sum of the coordinate weights is
close to its maximum, forcing most coordinates to lie in the high-weight regime.

We now verify the three required properties of $E_0$ (and the extension set $F_0$) in turn.

\noindent\textbf{$\cL$-solution-free.}
    The idea is that membership in $E_0$ forces most coordinates to sit in a narrow interval, so the linear form
    $\sum c_\ell x_\ell$ cannot wrap around to $0$.
    For $y\in E_0$, let $\delta^{(i)}(y)=1-\norm{y}^{(i)}_1\in[0,1]$. Then, from the definitions we know that
    \[
    \sum_{i=1}^n\delta^{(i)}(y)=n-\norm{y}_1\leq q\sqrt n+1.
    \]
    As $\norm{y}^{(i)}_1=\min\left\{\frac{q}{p_i}\,\overline{y^{(i)}},\;
\frac{q}{(q-1)p_i}\bigl(p_i-\overline{y^{(i)}}\bigr)\right\}
=1-\delta^{(i)}(y)$, the $i$-th coordinate of $y$ satisfies
    \[
    \frac{p_i\left(1-(q-1)\delta^{(i)}(y)\right)}{q}\leq\frac{p_i\left(1-\delta^{(i)}(y)\right)}{q}\leq \overline{y^{(i)}}\leq\frac{p_i\left(1+(q-1)\delta^{(i)}(y)\right)}{q}.
    \]
    Consequently, for $x_1,\dots,x_k\in E_0$, one has
    \[
    \frac{p_i}{q}\left(C-(q-1)\sum_{\ell=1}^k|c_\ell|\delta^{(i)}(x_\ell)\right)\leq \sum_{\ell=1}^kc_\ell \overline{x_\ell^{(i)}}\leq\frac{p_i}{q}\left(C+(q-1)\sum_{\ell=1}^k|c_\ell|\delta^{(i)}(x_\ell)\right).
    \]
    One checks that if $z\in\ZZ$, $y\in\ZZ_m$ and $z\equiv y^{(i)}\pmod{p_i}$, then
    \[
    \norm{y}^{(i)}_j\geq\min\left\{\frac{q}{jp_i}z,
    \frac{q}{(q-j)p_i}\left(p_i-z\right)\right\}.
    \]
    Indeed, when $z=\overline{y^{(i)}}$ this follows from definition, and otherwise the right hand side is non-positive. Thus using $C\ge 1$ and $q-C\ge 1$ we have
    \[
    \begin{split}
    &\norm{c_1x_1+\dots+c_kx_k}^{(i)}_{C}\\
    &\geq\min\left\{\frac{q}{Cp_i}\frac{p_i}{q}\left(C-(q-1)\sum_{\ell=1}^k|c_\ell|\delta^{(i)}(x_\ell)\right),\frac{q}{(q-C)p_i}\frac{p_i}{q}\left(q-C-(q-1)\sum_{\ell=1}^k|c_\ell|\delta^{(i)}(x_\ell)\right)\right\}\\
    &\geq1-(q-1)\sum_{\ell=1}^k|c_\ell|\delta^{(i)}(x_\ell).
    \end{split}
    \]
    Note that the above inequality is still valid if the right-hand side is negative as $\norm{\cdot}^{(i)}_C\ge 0$ by definition.
   
    Therefore,
    \begin{equation}\label{eq:sum_of_E0}
    \begin{split}
    \norm{c_1x_1+\dots+c_kx_k}_{C}&=\sum_{i=1}^n\norm{c_1x_1+\dots+c_kx_k}^{(i)}_{C}\geq n-(q-1)\sum_{\ell=1}^k|c_\ell|\sum_{i=1}^n\delta^{(i)}(x_\ell)\\
    &\geq n-(q-1)\sum_{\ell=1}^k|c_\ell|(q\sqrt n+1)=n-(q-1)D (q\sqrt n+1)\\
    &>n-q^2D \sqrt n>0.
    \end{split}
    \end{equation}
    Since $\overline{0^{(i)}}=0$ for every $i$, we have $\norm{0}^{(i)}_C=0$ and hence
$\norm{0}_C=\sum_{i=1}^n \norm{0}^{(i)}_C=0$.
We conclude from~\eqref{eq:sum_of_E0} that $c_1x_1+\dots+c_kx_k\neq0$ for all $x_1,\dots,x_k\in E_0$.

    \noindent\textbf{Large chromatic number.}
We shall embed the high-chromatic Cayley graph on $\ZZ_q^n$ in \cref{thm:cayley} into $\Cay(\ZZ_m,E_0)$ by discretizing each coordinate.
In particular, we show that the Cayley graph $\Cay\!\left(\ZZ_{q}^n,S\right)$ embeds into
    $\Cay\!\left(\ZZ_m,E_0\right)$, where
    \[
    S\coloneqq\left\{\bx\in\ZZ_q^n:\mathrm{d}(\bx,\bm{1})\leq q\sqrt n\right\}.
    \]

    To this end, define $f:\ZZ_q^n\to \prod_{i=1}^n \ZZ_{p_i}$ coordinate-wise by
    \[
    f(\bx)^{(i)}\coloneqq \left\lfloor \frac{\bx^{(i)}p_i}{q}\right\rfloor\in\ZZ_{p_i},
    \qquad i\in[n].
    \]
    Here we use $\bx^{(i)}$ for the $i$-th coordinate of a vector $\bx$. Since $p_i>q$, the map $j\mapsto \lfloor jp_i/q\rfloor$ is strictly increasing on $j\in\{0,1,\dots,q-1\}$. Hence the map $f$ is injective.

    It remains to show that adjacency is preserved.
    For an arc $(\by,\bx)$ in $\Cay\!\left(\ZZ_{q}^n,S\right)$, $\bx^{(i)}-\by^{(i)}$ is 1 for at least $n-q\sqrt n$ coordinates $i$. On each such $i$, one has
    \[
    \bigl(f(\bx)-f(\by)\bigr)^{(i)}=\left\lfloor\frac{\bx^{(i)}p_i}{q}\right\rfloor-\left\lfloor\frac{\by^{(i)}p_i}{q}\right\rfloor\in\left\{\left\lfloor\frac{p_i}{q}\right\rfloor,\left\lceil\frac{p_i}{q}\right\rceil\right\}\subseteq\ZZ_{p_i}.
    \]
    Therefore,
    \begin{equation}\label{eq:pi_large}
    \norm{f(\bx)-f(\by)}^{(i)}_1\geq\min\left\{\frac{q}{p_i}\left\lfloor\frac{p_i}{q}\right\rfloor,\frac{q}{(q-1)p_i}\left(p_i-\left\lceil\frac{p_i}{q}\right\rceil\right)\right\}>1-\frac{q}{p_i}>1-\frac{1}{n}.
    \end{equation}
    The last inequality is because $p_i>qn$ by the choice of $p_i$. Therefore,
    \[
    \norm{f(\bx)-f(\by)}_1>(n-q\sqrt n)\left(1-\frac{1}{n}\right)>n-q\sqrt n-1,
    \]
    hence $f(\bx)$ and $f(\by)$ are adjacent in $\Cay\!\left(\ZZ_m,E_0\right)$. Applying \cref{thm:cayley} with $p=q$, we conclude that
    \[
    \chi\!\left(\Cay(\ZZ_m,E_0)\right)\ge \chi\!\left(\Cay\!\left(\ZZ_{q}^n,S\right)\right)>\frac{\sqrt n}{q^3}>2t.
    \]

    \noindent\textbf{Linear size extension.}
    We define $F_0$ so that $c_1F_0-c_1F_0$ stays inside a ``small norm'' region,
    while $c_3E_0+\cdots+c_kE_0$ is forced into a ``large norm'' region by \eqref{eq:sum_of_E0};
    this makes the required intersection empty. The condition $|F_0|=\Theta(m)$ shall follow from a 2D Berry--Esseen argument.
    To be specific, let
    \[
    F_0 \coloneqq \left\{y \in \ZZ_m: \norm{c_1y}_{C}\leq n/2-q^2D \sqrt n ~\text{ and }~\norm{-c_1y}_{C}\leq n/2-q^2D \sqrt n\right\}.
    \]

    We first show that $|F_0|$ is linear in $m$. Let $y$ be uniformly random in $\ZZ_m$, and for each $i\in[n]$ define the two-dimensional random vector
\[
X_i\coloneqq\bigl(\norm{c_1y}^{(i)}_{C},~\norm{-c_1y}^{(i)}_{C}\bigr).
\]
Then, 
\[
X\coloneqq X_1+\cdots+X_n=\left(\sum_{i=1}^n \norm{c_1y}^{(i)}_{C},~\sum_{i=1}^n \norm{-c_1y}^{(i)}_{C}\right)
=\bigl(\norm{c_1y}_{C},~\norm{-c_1y}_{C}\bigr).
\]
To apply Corollary \ref{cor:2D-BE-convenient}, we first verify that
$\norm{X_i-\EE X_i}\leq1$ almost surely.
Since $\gcd(c_1,m)=1$ (as all $p_i>|c_1|$), multiplication by $c_1$ permutes $\ZZ_m$,
and hence $c_1y$ is uniformly distributed in $\ZZ_m$.
Therefore, one can compute
\[
\EE\left[\norm{c_1y}_C^{(i)}\right]=\EE\left[\norm{-c_1y}_C^{(i)}\right]=\frac12-\frac{r_i(q-r_i)}{2C(q-C)p_i^2},\quad\text{where $r_i \in [0, q-1]$ satisfies $r_i \equiv Cp_i\bmod q$}. 
\]
Consequently, let $\EE[X]=(\mu_0,\mu_0)$, then $\mu_0<n/2$, and
\[
\left(\norm{c_1y}_C^{(i)}-\EE\left[\norm{c_1y}_C^{(i)}\right]\right)^2
+
\left(\norm{-c_1y}_C^{(i)}-\EE\left[\norm{-c_1y}_C^{(i)}\right]\right)^2
\leq
2\left(\tfrac12+O\left(p_i^{-2}\right)\right)^2
<1,
\quad \forall\,y\in\ZZ_m,~i\in[n].
\]

Now it suffices to control the covariance matrix of $X$.
Since under the identification $\ZZ_m\cong\prod_{i=1}^n\ZZ_{p_i}$ the coordinates $y^{(i)}$ are independent and uniformly distributed,
the random vectors $X_1,\dots,X_n$ are mutually independent, and hence
$\mathrm{Cov}(X)=\sum_{i=1}^n \mathrm{Cov}(X_i)$.
Let $\Sigma_i$ be the covariance matrix of $X_i$. A direct computation gives
\[
\Sigma_i=
\begin{bmatrix}
    \frac{1}{12}+O\left(p_i^{-1}\right) & -\frac{7a^2-6a+1}{12(1-a)^2}+O\left(p_i^{-1}\right)\\
    -\frac{7a^2-6a+1}{12(1-a)^2}+O\left(p_i^{-1}\right) & \frac{1}{12}+O\left(p_i^{-1}\right)
\end{bmatrix},
\]
where $a\coloneqq C/q\in(0,1/3)$.
It has two eigenvalues
\[
\lambda_{i,1}=\frac{a(2-3a)}{6(1-a)^2}+O\left(\frac{1}{p_i}\right),\quad
\lambda_{i,2}=\frac{(1-2a)^2}{6(1-a)^2}+O\left(\frac{1}{p_i}\right),
\]
and the two eigenvalues are both positive.
Hence by Fact \ref{fact:eig-sum}, the covariance matrix $\Sigma_0$ of $X$ satisfies
\[
\lambda_{\min}(\Sigma_0)\geq\sum_{i=1}^n\min\{\lambda_{i,1},\lambda_{i,2}\}\geq\sigma^2n,\quad
\lambda_{\max}(\Sigma_0)\leq\sum_{i=1}^n\max\{\lambda_{i,1},\lambda_{i,2}\}\leq\sigma'^2n,
\]
where
\[
\sigma\coloneqq\min\left\{\frac{\sqrt{a(2-3a)}}{3(1-a)},\frac{1-2a}{3(1-a)}\right\},\quad
\sigma'\coloneqq\max\left\{\frac{\sqrt{a(2-3a)}}{2(1-a)},\frac{1-2a}{2(1-a)}\right\},
\]
and we use the fact that $p_i\gg1$.

Now let $A=\left[-\infty,\mu_0-2q^2D \sqrt {\mu_0}\right]^2$ and let $Z$ be the Gaussian vector $Z\sim N((\mu_0,\mu_0),\Sigma_0)$.
By the properties of Gaussian vectors
(Lemma \ref{lem:gauss-rect-const} with $r=2q^2D $, $c=\sigma^2$, $C=\sigma'^2$),
one has
\[
\PP(Z\in A)\geq\alpha\coloneqq\alpha_{2q^2D }\left(\sigma^2,\sigma'^2\right).
\]
Here $\alpha>0$ is a constant independent of $n$.

By Corollary \ref{cor:2D-BE-convenient}, we obtain
\[
\PP(X\in A)\geq \PP(Z\in A)-\frac{\beta}{\sigma^3}\frac{1}{\sqrt n}>\frac{\alpha}{2},
\]
where we use the fact that $n\gg1$.
In particular, for a uniformly random $y\in\ZZ_m$, the probability that
\[
\max\{\norm{c_1y}_C,\norm{-c_1y}_C\}\leq\mu_0-2q^2D \sqrt {\mu_0}< n/2-q^2D \sqrt n
\]
is at least $\alpha/2$,
which implies that $|F_0|=\Theta(m)$.

    It remains to verify the extension condition. On one hand, by triangle inequality of $\norm{\cdot}_C$ we have
    \[
    -c_1F_0-c_2F_0=c_1F_0-c_1F_0\subseteq\left\{y \in \ZZ_m: \norm{y}_{C}\leq n-2q^2D \sqrt n\right\}.
    \]
    On the other hand, from~\eqref{eq:sum_of_E0} we know that
    \[
    c_3E_0+\dots+c_kE_0\subseteq c_1E_0+\dots+c_kE_0\subseteq\left\{y \in \ZZ_m: \norm{y}_{C} > n-q^2D \sqrt n\right\}.
    \]
    Hence $F_0$ is indeed an extension of $E_0$ with size linear in $m$.
\end{proof}

\subsection{Proof of \cref{thm:a-implies-b-reduction}}
We now lift this construction to $\FF_p$ and prove \cref{thm:a-implies-b-reduction}. In the following we identify the elements in $\FF_p$ with $\{0,1,\dots,p-1\}$. Let $\varphi:\ZZ_m\to\FF_p$ be the map sending each residue class in $\ZZ_m$ to its representative in $\{-\frac{m-1}{2},\dots,\frac{m-1}{2}\}\subseteq\FF_p$,
and let $E = \varphi(E_0)$.
For $p\gg m$, this preserves (i) solution-freeness and (ii) the induced Cayley subgraph
on $\{-\frac{m-1}{2},\dots,\frac{m-1}{2}\}$, hence preserves the chromatic lower bound.
We then add a large extension set $F$ to form $A=E\sqcup F$ while keeping $\cL$-solution-freeness.
Let us first verify that $E$ satisfies all three conditions in Lemma \ref{lem:small-generate-large-cayley} with $\ZZ_m$ replaced by $\FF_p$ when $p \gg m$.
    
    \noindent\textbf{$\cL$-solution-free.}
    If $E$ has a distinct solution $(x_1,\dots,x_k)$ to $\cL$ in $\FF_p$, we view each $x_i$ as its representative in $\{-\frac{m-1}{2},\dots,\frac{m-1}{2}\}\subseteq\ZZ$. Then
\[
\Bigl|\sum_{i=1}^k c_i x_i\Bigr|\le \sum_{i=1}^k |c_i|\,|x_i|\le D(m-1) < p.
\]
Hence $\sum_{i=1}^k c_i x_i\equiv 0\pmod p$ forces $\sum_{i=1}^k c_i x_i=0$ over $\ZZ$.
Applying $\varphi^{-1}$ gives $c_1\varphi^{-1}(x_1)+\dots+c_k\varphi^{-1}(x_k)=0$ in $\ZZ_m$, contradicting that $E_0$ is $\cL$-solution-free.
Thus $E$ is $\cL$-solution-free.

    \noindent\textbf{Large chromatic number.}
    Let $M\coloneqq\{0,1,\dots,\tfrac{m-1}2\}$.
The subgraphs of $\Cay(\ZZ_m,E_0)$ and $\Cay(\FF_p,E)$ induced on $M$ are isomorphic.
Indeed, for $x_1,x_2\in M$, we have that $x_1$ and $x_2$ are adjacent in $\Cay(\FF_p,E)$ if and only if there exists $a\in E$
such that
\[
x_1-x_2\equiv \pm a \pmod p.
\]
Since $|x_1-x_2|\le \frac{m-1}{2}$ and $a\in E\subseteq\{-\frac{m-1}{2},\dots,\frac{m-1}{2}\}$, we have
\[
|x_1-x_2\mp a|\le |x_1-x_2|+|a|<m<p,
\]
and hence the congruence holds if and only if $x_1-x_2=\pm a$ over $\ZZ$.
Equivalently, $x_1$ and $x_2$ are adjacent in $\Cay(\ZZ_m,E_0)$.

By symmetry, the subgraphs of $\Cay(\ZZ_m,E_0)$ induced on $M$ and $-M$ are isomorphic, and since $M\cup(-M)=\ZZ_m$, one has
\[
\chi(\Cay(\FF_p,E))\geq\chi(\Cay(\ZZ_m,E_0)[M])\geq\frac12\chi(\Cay(\ZZ_m,E_0))>t.
\]
    
   \noindent\textbf{Linear size extension.}
Let
\[
I_p\coloneqq\Bigl[\frac{p}{D +1},\,\frac{p}{D }\Bigr]\cap\FF_p,
\]
and define
\[
F\coloneqq\{x\in I_p:~x \bmod m \in F_0\}.
\]
Then $|F|=\Theta(p)$ since $p\gg m$ and $|F_0|=\Theta(m)$.

We claim that $F$ is an extension of $E$ in $\FF_p$, i.e.
\[
(-c_1F-c_2F)\cap (c_3E+\cdots+c_kE)=\varnothing.
\]
Indeed, if the intersection were non-empty, then there exist $x_1,x_2\in F$ and $x_3,\dots,x_k\in E$ such that
\[
c_1x_1+c_2x_2+c_3x_3+\cdots+c_kx_k\equiv 0 \pmod p.
\]
Now we view $x_1,\dots,x_k\in I_p\cup\{-\frac{m-1}{2},\dots,\frac{m-1}{2}\}$ as integers. Using $c_2=-c_1$ and $x_1,x_2\in I_p$, we have $|x_1-x_2|\le \frac{p}{D (D +1)}$, hence
\[
\Bigl|\sum_{i=1}^k c_i x_i\Bigr|
\le |c_1|\,|x_1-x_2|+\sum_{i=3}^k |c_i|\,|x_i|
< D \cdot \frac{p}{D (D +1)}+D m
< p.
\]
Therefore the above congruence forces $\sum_{i=1}^k c_i x_i=0$ over $\ZZ$, and hence also modulo $m$.
Thus in $\ZZ_m$, $(-c_1F_0-c_2F_0)\cap (c_3E_0+\cdots+c_kE_0)\neq\varnothing$, contradicting that $F_0$ is an extension of $E_0$.

    Now we take
    \[
    A = E \sqcup F.
    \]
Then $A$ has size linear in $p$, and $\Cay(\FF_p, A)$ contains the subgraph $\Cay(\FF_p, E)$, which has chromatic number at least $t$.
It remains to check that $A$ is $\cL$-solution-free.

The idea is that any $\cL$-solution in $A=E\sqcup F$ would induce a nontrivial overlap between
a sumset coming from $F$ and a sumset coming from $E$. The interval choice for $F$ forces the total coefficient
sum over $F$-indices to be $0$, leaving only the degenerate possibilities ruled out by construction.

    Suppose to the contrary that there exist distinct $x_1,\dots,x_k\in A$ which form a solution to $\cL$.
    Let $J\subseteq[k]$ denote the set of indices for which $x_j\in F$, so that $x_j\in E$ for all $j\notin J$. Since $E$ is $\cL$-solution-free, $J\neq \varnothing$.
    Rewriting the equation $\cL$ by separating the contributions from $F$ and $E$, we obtain 
    \begin{equation}\label{eq:if}
        -\sum_{j\in J} c_j x_j\equiv \sum_{j\notin J} c_j x_j\pmod p.
    \end{equation}
    Equivalently, we have
    \[
    -\left(\sum_{j\in J}c_jF\right)\cap \sum_{j\notin J}c_jE\neq\varnothing,
    \]
   where we adopt the convention that $\sum_{j\notin J} c_j E=\{0\}$ when $J=[k]$.
 Since $F\subseteq\bigl[p/(D +1),\,p/D \bigr]$ by definition, we know that
\[
\sum_{j\in J}c_jF\subseteq\left[\sum_{j\in J,c_j>0}\frac{|c_j|p}{D +1}-\sum_{j\in J,c_j<0}\frac{|c_j|p}{D },\sum_{j\in J,c_j>0}\frac{|c_j|p}{D }-\sum_{j\in J,c_j<0}\frac{|c_j|p}{D +1}\right].
\]
This means that the set $\sum_{j\in J}c_jF$ is contained in a \emph{short} interval of length $O(p/D ^2)$: varying the $x_j\in F$ only perturbs the sum by a relatively small amount. 

Set 
$$s\coloneqq\sum_{j\in J}c_j.$$ 
We distinguish two cases depending on whether $s$ equals zero or not.

\noindent\textbf{Case 1. $s\ne 0$.} Since $x_j\in F\subseteq\left[\frac{p}{D +1},\frac{p}{D }\right]$ for $j\in J$, writing
$b\coloneqq\frac{p}{D }$ and $x_j=b-u_j$ with $0\le u_j\le b-\frac{p}{D +1}$, we have
\[
\sum_{j\in J} c_j x_j = sb-\sum_{j\in J} c_j u_j
\quad\text{and}\quad
\Bigl|\sum_{j\in J} c_j u_j\Bigr|\le \sum_{j\in J}|c_j|\Bigl(b-\frac{p}{D +1}\Bigr)\le \frac{p}{D +1}.
\]
If $s\ge 1$, then $\sum_{j\in J} c_j x_j \ge b-\frac{p}{D +1}=\frac{p}{D (D +1)}$;
if $s\le -1$, then similarly $\sum_{j\in J} c_j x_j \le -\frac{p}{D (D +1)}$.
We always have
\[
\frac{p}{D (D +1)}\le\Bigl|\sum_{j\in J} c_j x_j\Bigr|\le \frac{(D-1)p}{D}.
\]
Thus the contribution from the indices in $J$ cannot be close to $0$ modulo $p$: it is separated from $0$ by a gap of size $\gtrsim p/D ^2$. On the other hand, since $E\subseteq[-\frac{(m-1)}2,\frac{(m-1)}2]$, we have $\bigl|\sum_{j\notin J} c_j x_j\bigr|\le D m$. In other words, the $E$-part is \emph{tiny} compared to the $F$-part once $p\gg D ^2 m$. This contradicts~\eqref{eq:if}.

\noindent\textbf{Case 2. $s=0$.} By the assumption that no subset of size at least $3$ has
coefficient sum zero, this forces $J=\{u,v\}$ for some $u,v\in[k]$. Hence, we know that $\{u,v\}\cap\{1,2\}\neq\varnothing$ by previous discussion. Without loss of generality, suppose $u=1$. Then $c_v=-c_1=c_2$, and from~\eqref{eq:if} we have
\[
-c_1 x_1-c_2 x_v=-\sum_{j\in J} c_j x_j\equiv \sum_{j\notin J} c_j x_j=\sum_{j=3}^kc_jx_j-c_vx_v+c_vx_2\pmod p.
\]
By the definition of $J$, this contradicts that $F$ is an extension of $E$. 

To conclude, $A=E\sqcup F$ is $\cL$-solution-free. This completes the proof of \cref{thm:a-implies-b-reduction}.

\section{Concluding remarks}\label{sec:conclude}

We have classified the homogeneous linear equations with vanishing chromatic threshold: $\delta_\chi(\cL)=0$ exactly when some zero-sum subcollection contains at least three coefficients. The two directions expose complementary mechanisms. A balanced subequation of length at least three yields a Fourier--Bohr coloring, whereas a lone canceling pair permits dense solution-free generators with arbitrarily large chromatic number. The generalized Kneser construction and the resulting Hamming-ball theorem on $\ZZ_p^n$ are also of independent interest. Their application to separating topological and measurable recurrence in all countably infinite abelian groups requires additional dynamical machinery and is developed in an upcoming work~\cite{separatingnote}.

A natural next step is to replace the prime cyclic groups by broader families of finite abelian groups. The Fourier--Bohr argument suggests an extension when the relevant coefficients act invertibly, but the required supersaturation statement and the dependence on group torsion must be checked carefully. The converse construction is more rigid: it uses a product-group model, an odd-characteristic Kneser obstruction, and an ordered lift to a large prime cyclic group. A group-uniform classification therefore appears to require genuinely new input.

Our theorem determines only whether the threshold vanishes. The first quantitative case already appears difficult.

\begin{problem}
Determine the exact value of $\delta_\chi(x+y=z)$.
\end{problem}

\paragraph{Syndetic sets and VC dimension.}

Two further thresholds measure different forms of global structure. As in the
main theorem, we restrict throughout to prime cyclic groups. For
$A\subseteq\FF_p$, define its \emph{translative covering number} by
\[
\tau(A)\coloneqq
\min\bigl\{|T|:T\subseteq\FF_p,\ A+T=\FF_p\bigr\}.
\]
Thus $A$ is syndetic with covering number at most $C$ precisely when
$\tau(A)\le C$.

Let
$\mathcal N_A\coloneqq\{x+A:x\in\FF_p\}$
be the translate set system associated with $\Cay(\FF_p,A)$, and write
$\textup{VC}(\Cay(\FF_p,A))
\coloneqq
\textup{VC}(\mathcal N_A).$
We define
\begin{align*}
\delta_{\textup{syn}}(\cL)
&\coloneqq
\inf\Bigl\{d>0:\ \exists C=C(d,\cL)\ \text{such that, for every prime $p$
and every $\cL$-solution-free $A\subseteq\FF_p$}\\
&\hspace{33mm}
\text{with $|A|\ge dp$, one has $\tau(A)\le C$}\Bigr\},\\
\delta_{\textup{VC}}(\cL)
&\coloneqq
\inf\Bigl\{d>0:\ \exists C=C(d,\cL)\ \text{such that, for every prime $p$
and every $\cL$-solution-free $A\subseteq\FF_p$}\\
&\hspace{33mm}
\text{with $|A|\ge dp$, one has
$\textup{VC}(\Cay(\FF_p,A))\le C$}\Bigr\}.
\end{align*}

To compare these parameters with the ordinary density problem, define
\[
\pi_{\mathrm{Roth}}(\cL)
\coloneqq
\limsup_{\substack{p\to\infty\\ p\text{ prime}}}
\frac{1}{p}\max\bigl\{|A|:A\subseteq\FF_p
\text{ is $\cL$-solution-free}\bigr\}.
\]
Although the syndetic and VC-dimension thresholds appear to impose stronger
structural requirements, they coincide exactly with this ordinary extremal
density.

\begin{proposition}\label{prop:syn-vc-roth}
For every homogeneous linear equation $\cL$,
\[
\delta_{\textup{syn}}(\cL)
=
\delta_{\textup{VC}}(\cL)
=
\pi_{\mathrm{Roth}}(\cL).
\]
Consequently,
$\delta_{\textup{syn}}(\cL)=0
\ \Longleftrightarrow \ 
\delta_{\textup{VC}}(\cL)=0 \Longleftrightarrow \ 
\sum_{i=1}^k c_i=0.$
\end{proposition}

\begin{proof}
We first show that
$\delta_{\textup{syn}}(\cL)
\le
\delta_{\textup{VC}}(\cL).$
Fix $d>\delta_{\textup{VC}}(\cL)$. There exists
$D=D(d,\cL)$ such that every $\cL$-solution-free set
$A\subseteq\FF_p$ with $|A|\ge dp$ satisfies
$\textup{VC}(\mathcal N_A)\le D.$
The set systems
$\{x+A:x\in\FF_p\}$
and 
$\{x-A:x\in\FF_p\}$
have the same VC dimension, since the negation map is a bijection of their
ground sets. Every member of the latter family has size at least $dp$.
The standard $\varepsilon$-net theorem, applied with $\varepsilon=d$, therefore
gives a set $T\subseteq\FF_p$ of size $O_{D,d}(1)$ meeting every translate
$x-A$. For each $x\in\FF_p$, choose
$t\in T\cap(x-A).$
Then $t=x-a$ for some $a\in A$, and hence $x=a+t$. It follows that
\[
A+T=\FF_p.
\]
Thus $\tau(A)=O_{D,d}(1)$, proving the claimed inequality.

We next prove that
$\delta_{\textup{VC}}(\cL)
\le
\pi_{\mathrm{Roth}}(\cL).$
Fix $d>\pi_{\mathrm{Roth}}(\cL)$. By the definition of the limsup, for all
sufficiently large primes $p$ there is no $\cL$-solution-free set
$A\subseteq\FF_p$ with $|A|\ge dp$. For the finitely many remaining primes,
every set system on $\FF_p$ has VC dimension at most $p$. Hence there exists
a constant depending only on $d$ and $\cL$ that bounds the VC dimension of
every $\cL$-solution-free set of density at least $d$. Therefore
$\delta_{\textup{VC}}(\cL)
\le
\pi_{\mathrm{Roth}}(\cL).$

It remains to prove
$\pi_{\mathrm{Roth}}(\cL)
\le
\delta_{\textup{syn}}(\cL).$
Fix $d<\pi_{\mathrm{Roth}}(\cL)$, and choose constants
$d<\beta<\pi_{\mathrm{Roth}}(\cL)$
and 
$\theta\in(0,1)$
such that
$\theta\beta>d.$
There are infinitely many primes $p$ for which one can find an
$\cL$-solution-free set $B\subseteq\FF_p$ satisfying
$|B|\ge\beta p.$
Fix an arbitrary integer $K\ge1$, and choose such a prime $p$ sufficiently
large. Form a random subset $A\subseteq B$ by retaining each element of $B$
independently with probability $\theta$. Since
$\mathbb E|A|=\theta|B|\ge\theta\beta p>dp,$
a Chernoff bound gives
$\mathbb P(|A|\ge dp)=1-o(1).$
Moreover, $A$ is automatically $\cL$-solution-free because
$A\subseteq B$.

We claim that
\[
\mathbb P(\tau(A)\le K)=o(1).
\]
Fix $T\subseteq\FF_p$ with
$1\le r\coloneqq|T|\le K$.
We may greedily choose a set $X\subseteq\FF_p$ with
$|X|\ge \frac{p}{r^2}$
such that the translates $\{x-T:x\in X\}$ are pairwise disjoint. Indeed,
a fixed translate $x-T$ intersects at most
$|T-T|-1<r^2$ other translates of this form. If $A+T=\FF_p$, then
\[
A\cap(x-T)\ne\varnothing
\qquad\text{for every }x\in X.
\]
Since the sets $x-T$, $x\in X$, are pairwise disjoint and $A$ is obtained
by independently retaining elements of $B$, these events are independent.
Moreover, for every $x\in X$,
\[
\mathbb P\bigl(A\cap(x-T)\ne\varnothing\bigr)
=
1-(1-\theta)^{|B\cap(x-T)|}
\le
1-(1-\theta)^r.
\]
Consequently,
\[
\mathbb P(A+T=\FF_p)
\le
\bigl(1-(1-\theta)^r\bigr)^{p/r^2}
\le
\exp\left(-\frac{(1-\theta)^r}{r^2}p\right)
\le
\exp\left(-\frac{(1-\theta)^K}{K^2}p\right).
\]
There are at most $Kp^K$ choices for a subset $T\subseteq\FF_p$ of size at
most $K$. Hence, by the union bound,
\[
\mathbb P(\tau(A)\le K)
\le
Kp^K
\exp\left(-\frac{(1-\theta)^K}{K^2}p\right)
=o(1).
\]
Thus, for all sufficiently large primes in the chosen sequence, there exists
an $\cL$-solution-free set $A\subseteq\FF_p$ such that
$|A|\ge dp$
and 
$\tau(A)>K.$
Since $K$ was arbitrary, no uniform covering-number bound is possible at
density $d$. Hence $d$ is not admissible in the definition of
$\delta_{\textup{syn}}(\cL)$, and therefore
$\delta_{\textup{syn}}(\cL)
\ge
\pi_{\mathrm{Roth}}(\cL).$

Combining the three inequalities proves
$\delta_{\textup{syn}}(\cL)
=
\delta_{\textup{VC}}(\cL)
=
\pi_{\mathrm{Roth}}(\cL).$
The final equivalence follows from \cref{thm:roth}.
\end{proof}

Thus, unlike the chromatic threshold, the syndetic and VC-dimension
thresholds do not produce new vanishing classes: both collapse to Roth's
ordinary extremal density.

\bibliographystyle{abbrv}
\bibliography{1.bib}

\end{document}